\documentclass[11pt]{amsart}

\usepackage[utf8]{inputenc}
\usepackage[english]{babel}
\usepackage{appendix}
\usepackage{tikz}
\usepackage[right=2.5cm,bottom=2.5cm,footskip=30pt,left=2.5cm]{geometry}
\usepackage{comment}
\usepackage{graphicx}
\usepackage{enumerate}
\usepackage[T1]{fontenc}
\usepackage{color}
\usepackage{mathrsfs}
\usepackage{mathtools}
\usepackage{enumitem}
\usepackage{amsmath, amsthm, amssymb}
\usepackage{accents}
\usepackage{csquotes}
\usepackage{hyperref}

\theoremstyle{plain}
\newtheorem{theorem}{Theorem}[section]

\newtheorem{lemma}[theorem]{Lemma}
\newtheorem{proposition}[theorem]{Proposition}

\newtheorem{definition}[theorem]{Definition}
\newtheorem{remark}[theorem]{Remark}

\theoremstyle{remark}
\numberwithin{equation}{section}
\usepackage[
  backend=biber,
  style=numeric,
  sorting=nyt,
  giveninits=true,
  terseinits=false,
  maxbibnames=99, maxcitenames=99
]{biblatex}

\mathtoolsset{showonlyrefs,showmanualtags}

\author{Virginia Agostiniani}
\address{ Università degli Studi di Trento}
\email{virginia.agostiniani@unitn.it}

\author{Christina Karapiperi}
\address{Università degli Studi di Trento}
\email{christina.karapiperi@unitn.it}
\author{Lorenzo Mazzieri}
\address{Università degli Studi di Trento}
\email{lorenzo.mazzieri@unitn.it}
\title[Riemannian Penrose Inequality for Manifolds with Corners]{Riemannian Penrose Inequality for Manifolds with Corners via Non-Linear Potential Theory}

\begin{document}
 \begin{abstract}
We present a new proof of the Positive Mass Theorem  and the Riemannian Penrose Inequality for three-dimensional asymptotically flat Riemannian manifolds whose metrics fail to be $C^1$ across a hypersurface $\Sigma$, first proven in \cite{miao_cornerPMT} and \cite{Miao_McCormik_Penrose_like_Inequality}, respectively. Unlike these approaches, ours recovers these results directly, without relying on their original formulations for smooth metrics. The proofs are based on a unified argument which applies to both theorems. We achieve this by establishing an approximate monotonicity for the quantity introduced in \cite{agostiniani_riemannianpenroseinequalitynonlinear_2022}, employing the approximation scheme in \cite{miao_cornerPMT}, for metrics with $C^{2,\alpha}$ regularity up to $\Sigma$.
\end{abstract}  
\maketitle
\noindent

\smallskip

\noindent \underline{\smash{Keywords}}: Positive Mass Theorem, Riemannian Penrose Inequality, monotonicity formulas, geometric inequalities, potential theory.

\section{Introduction}
The Positive Mass Theorem was first proven by Schoen and Yau  \cite{schoen_proofpositivemassconjecture_1979} for smooth asymptotically flat manifolds with non-negative scalar curvature using minimal surfaces. Alternative proofs were later provided by various authors (\cite{agostiniani_greenfunctionproofpositive_2024}, \cite{bray_harmonicfunctionsmassdimensional_2022}, \cite{witten_newproofpositiveenergy_1981}, to name a few) in a smooth setting. In \cite{huisken_inversemeancurvatureflow_2001}, Huisken and Illmanen proved the Riemannian Penrose Inequality using the Weak Inverse Mean Curvature Flow, see also \cite{agostiniani_riemannianpenroseinequalitynonlinear_2022} for a proof of the Riemannian Penrose Inequality via a monotonicity formula in non linear potential theory.

In 2002, a proof for the positive mass theorem for manifolds that admit a hypersurface on which the metric fails to be $C^1$ was provided by Miao, \cite{miao_cornerPMT}. In that paper, Miao used a smoothing technique and a conformal deformation to obtain a family of metrics with nonnegative scalar curvature. A comparison of Miao's and our method will be provided later in the manuscript. Later, Lee and LeFloch introduced a notion of distributional scalar curvature and eventually proved the Positive Mass Theorem in this setting, \cite{lee_positivemasstheoremmanifolds_2015}. For other approaches in this direction see also \cite{Li_2020_Distributional_Positive_Mass_Theorem_Non_Spin}, \cite{mcferon2012positive}.
After having established the non-negativity for the ADM-mass for manifolds with corners, it is natural to ask whether the Riemannian Penrose Inequality can also be extended to these manifolds. The statement has been proven in \cite{Miao_McCormik_Penrose_like_Inequality} by a repetition of the argument in \cite{Miao_localized_Riemannian_Penrose}, while the rigidity was shown in \cite{Miao_Lu_Rigidity_Penrose},\cite{Shi_Wang_Yu_Rigidity_Penrose_Corners}.

We will present a unified method that suffices to obtain both the Positive Mass Theorem and the Riemannian Penrose Inequality with corners. During the final stage of preparation of this manuscript, a similar approach for the Positive Mass Theorem for $C^0$ metrics appeared in \cite{mazurowski2026positivemasstheoremcontinuous}. \\ \\
Let $M$ be an oriented, 3-dimensional, smooth, differentiable manifold. We assume that there exists an open bounded domain $\Omega \subset M$ so that $M\setminus \Omega$ is diffeomorphic to $\mathbb{R}^3$ without a ball
 and $\Sigma$ to be the smooth boundary of $\Omega$.
\begin{definition}
\label{metricwithcorners}
    A metric $G$ admitting corners along a hypersurface $\Sigma$ is defined to be pair of $(g_{-}, g_{+})$, where $g_{-}$ and $g_{+}$  metrics on $\overline{\Omega}$ and $M \setminus {\Omega}$, respectively, such that they are $C^{2,\alpha}$ up to the boundary and induce the same metric on it.
\end{definition}
\begin{definition}\label{asymptotically_flat_with_corners}
    We say that the manifold $M$ equipped with the metric $G$ is asymptotically flat if the manifold $(M \setminus \Omega, g_{+})$ is asymptotically flat. By this we understand that the ADM-mass of $M$ with the metric $G$, $m_{ADM}(M, G)$, is the ADM-mass of $(M \setminus \Omega, g_{+})$.
\end{definition}
In this paper we will be using  $C^{1}_\tau$-asymptotic flatness, see Appendix \ref{AppA}.
\begin{theorem}[Positive Mass Theorem with Corners, {\cite[Theorem 1]{miao_cornerPMT}}]
\label{MIAO}Let $(M, G)$ be an oriented, 3-dimensional, smooth, complete, connected, $C^1_{\tau}$-asymptotically flat Riemannian manifold $M$, $\tau>\frac{1}{2}$, with no boundary and $\mathbb{H}_2(M;\mathbb{Z})=\{0\}$. We assume that there exists a bounded domain $\Omega \subset M$ with smooth boundary $\Sigma$ such that $M \setminus \overline{\Omega}$ is diffeomorphic to $\mathbb{R}^3 \setminus \overline{\mathbb{B}^3}$ and $G=(g_{-},g_{+})$ is a metric admitting corners along $\Sigma$ as in Definition \ref{metricwithcorners}. Suppose that the scalar curvatures $R_{-}$ and ${R
}_+$ of $g_{-}$ and $g_{+}$ are non-negative in $\overline{\Omega}$ and $M\setminus \Omega$, respectively. Assume furthermore that
\begin{equation}
\label{H} \tag{H}
    H_- \geq H_+,
\end{equation}
    where $H_-$ and $H_+$ represent the mean curvature of $\Sigma$ in $(\overline{\Omega},g_{-})$ and $(M\setminus \Omega,g_{+})$ respectively, computed with respect to the unit normal vectors pointing towards the unbounded region. Then the ADM-mass of $M$ is nonnegative.
\end{theorem}
For the rigidity statement of the Positive Mass Theorem with corners we refer to \cite[Theorem 2]{miao_cornerPMT}. 
The second main theorem is the following.
\begin{theorem}[Riemannian Penrose Inequality with Corners for a single black hole, {\cite[Proposition 3.1]{Miao_McCormik_Penrose_like_Inequality}}]\label{Riemannian_Penrose_Corners}Let $(M, G)$ be an oriented, three-dimensional, smooth, complete, connected $C^{1}_{\tau}$- asymptotically flat Riemannian manifold, $\tau>\frac{1}{2}$, with smooth, compact, connected, boundary $\partial M$. Assume furthermore that $\mathbb{H}_2(M,\partial M;\mathbb{Z})=\{0\}$, that $\partial M$ is the unique closed minimal surface in $(M,G)$ and that there exists a domain $\Omega \subset M$ with smooth boundary
$\Sigma$, with $\partial M \subset \Omega$, such that $M \setminus \Omega$ is diffeomorphic to $\mathbb{R}^3 \setminus \mathbb{B}^3$ and $G=(g_{-}, g_{+})$ is
a corner metric along $\Sigma$ in the sense of Definition \ref{metricwithcorners}. The scalar curvatures $R_{-}$ and ${R
}_+$ of $g_{-}$ and $g_{+}$ are non-negative in $\overline{\Omega}$ and $M\setminus \Omega$, respectively and Condition \eqref{H} is satisfied. Then, the ADM-mass satisfies 
\begin{equation}
    m_{ADM}(M,G)\geq \sqrt{\frac{\left|\partial M\right|}{16 \pi}}.
\end{equation}
\end{theorem}
For the corresponding rigidity statement we refer the reader to \cite{Miao_Lu_Rigidity_Penrose}.
Miao's approach in \cite{miao_cornerPMT} for the proof of the Positive Mass Theorem with Corners is based on a smoothing argument for the metric around $\Sigma$, followed by a conformal deformation, in order to obtain a family of metrics of nonnegative scalar curvature. The ADM-mass of these metrics converges to the one of the original metric. The argument is then concluded by the use of the Positive Mass Theorem for smooth metrics. The argument for the Riemannian Penrose Inequality with corners in \cite{Miao_McCormik_Penrose_like_Inequality} follows a very similar path and eventually employs the statement for the Riemannian Penrose Inequality for smooth metrics.
\\ \\
Our proof is a combination of the aforementioned approximation scheme, see for completeness \cite[Appendix A]{Miao_McCormik_Penrose_like_Inequality} and an almost-monotonicity argument of a suitable quantity, used in \cite{agostiniani_riemannianpenroseinequalitynonlinear_2022} to prove the smooth Riemannian Penrose Inequality. We begin the analysis by considering two level sets of a potential that enclose the smoothing region introduced in Miao’s construction. The almost-monotonicity is obtained by exploiting this geometric setup together with the nonnegativity of the scalar curvature outside the smoothing region and a Caccioppoli-type inequality valid there, Lemma \ref{p_uniform_gradient_bound}.
\begin{figure}[h]
\centering
\resizebox{0.4\textwidth}{!}{
\begin{tikzpicture}[scale=2][h]\label{construction_picture}
\colorlet{myhighlight}{red!25}
\def\mycurve{
  (0,0)
    .. controls +(20:1.4)  and +(-110:1.2) .. (2.2,1.3)
    .. controls +(70:1.1)  and +(0:1.6)    .. (0.8,3.6)
    .. controls +(180:1.2) and +(110:1.0)  .. (-1.6,2.2)
    .. controls +(-70:1.3) and +(200:1.1)  .. cycle
}

\draw[line width=1cm, red] \mycurve;
\draw[line width=0.8cm, white] \mycurve;
\draw[line width=0.8cm, myhighlight] \mycurve;

\draw[line width=0.1cm, black] \mycurve;

\coordinate (o) at (0.5,2);

\draw[line width=2pt, blue] (o) ellipse [x radius=1.63cm, y radius=1cm];

\fill (2.13,2) circle[radius=1.8pt];

\draw[line width=2pt, blue, rotate around={22:(o)}] (o) ellipse [x radius=3cm, y radius=2.165cm];
\fill (-0.45,-0.28) circle[radius=1.8pt];

\node[above right, font=\huge] at (1,3.57) {
$\Sigma$};
\end{tikzpicture}}
\caption{Enclosing smoothing region between two level sets.}
\label{fig:my-picture}
\end{figure}
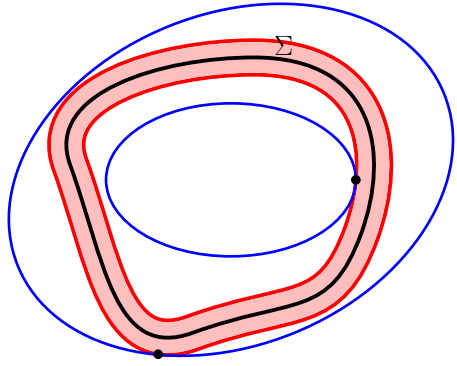
Our approach does not rely on the statements for the Positive Mass Theorem nor the Riemannian Penrose Inequality for smooth metrics. At this point, it is important to note that in the work of \cite{miao_cornerPMT} the regularity of the metrics $g_{-}$ and $g_{+}$ is $C^2$ up to $\Sigma$ and $C^{2,\alpha}_{loc}$ away from it. Our proof requires the use of Sard's Theorem and Bochner's formula, which leads into asking more regularity for the metric, namely $C^{2,\alpha}$ up to $\Sigma$. In addition, our proof requires that the manifold has trivial second homology group. Finally, note also that the proof in \cite{miao_cornerPMT} holds for any dimension for which the classical Positive Mass Theorem holds, while the proof of the Riemannian Penrose Inequaity with corners in \cite{Miao_localized_Riemannian_Penrose},\cite{Miao_McCormik_Penrose_like_Inequality} holds for dimension $3\leq n \leq 7$. Instead, our proof applies only to the case $n=3$ due to the need of the Gauss-Bonnet Theorem.\\

We recall that Condition \eqref{H} is related to the nonnegativity of the scalar curvature in the distributional sense. This concept of scalar curvature has been introduced in \cite{lee_positivemasstheoremmanifolds_2015} and it is defined as follows.{
    Let $\widetilde{M}$ be a smooth manifold equipped with a smooth background metric $\widetilde{h}$. Given any Riemannian metric $g$ with $L_{loc}^{\infty} \cap W_{loc}^{1,2}$ regularity and inverse $g^{-1} \in L_{loc}^{\infty},$ the \emph{distributional scalar curvature} $R_g$ is defined as
    \begin{equation}
    \label{distributionalscalarcurvature}
        \langle R_g, \phi \rangle := \int_{\widetilde{M}} \Big(-V\cdot \tilde{\nabla}(\phi \frac{d\mu_g}{d\mu_{\widetilde{h}}}) + F \frac{d\mu_g}{d\mu_{\widetilde{h}}}\phi \Big) d\mu_{\widetilde{h}},
    \end{equation}
    for every compactly supported smooth test function $\phi: \widetilde{M} \to \mathbb{R}$, where
    \begin{equation}
        V=V^k:= g^{ij}\Gamma_{ij}^k - g^{ik} \Gamma_{ji}^j,
    \end{equation}
    \begin{equation}
        F:= \widetilde{R}-\widetilde{\nabla}_k g^{ij}\Gamma_{ij}^k +\widetilde{\nabla}_k g^{ik}\Gamma_{ji}^j +g^{ij}(\Gamma_{kl}^k\Gamma_{ij}^l-\Gamma_{jl}^k\Gamma_{ik}^l),
    \end{equation}
    where $\widetilde{\nabla}$, $\widetilde{R}$ and $\Gamma$ denote the Levi-Civita connection, the scalar curvature and the Christoffel symbols of the metric $\widetilde{h}$ respectively. Note also that one has the regularity 
\[
\Gamma \in L^2_{\mathrm{loc}}, 
\qquad 
V \in L^2_{\mathrm{loc}}, 
\qquad 
\text{and}
\qquad 
F \in L^1_{\mathrm{loc}},
\]
and
\[
\frac{d\mu_g}{d\mu_h}
\in
L^\infty_{\mathrm{loc}} \cap W^{1,2}_{\mathrm{loc}}
\]
is the density of \(d\mu_g\) with respect to \(d\mu_h\).
In particular, in \cite[Proposition 5.1]{lee_positivemasstheoremmanifolds_2015} it is shown that if Condition \eqref{H} holds, then the distributional scalar curvature is nonnegative. A perusal of their proof shows that also the converse holds true and therefore in the present context the two conditions are equivalent. In the same paper, they also prove a Positive Mass Theorem in this more generic context.\\
In addition, note that if the metric in the corner is $C^1$, then Condition \eqref{H} is trivially satisfied with equality. This condition, even though stronger, does not simplify the argument of our proof. This is because Condition \eqref{H} combined with Miao's smoothing argument will be used to ensure that the negative part of the scalar curvature is uniformly controlled in a neighborhood of $\Sigma$,
\begin{equation*}
    R_\varepsilon \geq -C \qquad \text{or, equivalently,}\qquad \|R_\varepsilon^-\|\leq C.
\end{equation*}In general, even if the scalar curvatures of $g_{-}$ and $g_{+}$ are nonnegative, one cannot expect the nonnegativity of the mass for an asymptotically flat metric $G=(g_{-}, g_{+})$ admitting corners, when condition $\eqref{H}$ does not hold. For example consider the following metrics for $M\setminus \overline{\Omega}$ and $\Omega$
\begin{align}
    g_{+}=& \frac{dr^2}{1-\frac{2M}{r}}+r^2 g_{\mathbb{S}^2}, \hspace{0.35cm} r\geq R, \\
    g_{-}= & dr^2 + r^2g_{\mathbb{S}^2}, \hspace{1.0cm} r\leq R,
\end{align}
respectively, and fix $M<0$. Computing the corresponding mean curvatures for the hypersurface $\{r=R\}$, we obtain $H_+=\frac{2}{R_{-}}
\sqrt{1-\frac{2M}{R_{-}}}$, $H_-=\frac{2}{R_{-}}$.}
    \subsection{Physical Motivation of Nonnegative Scalar Curvature Corner Condition (H)}\label{discussion_on_H_condition}The importance of studying a lower bound for the mass in presence of corners arises from a phenomenon in astrophysics called Shells of Matter. In this phenomenon matter gets accumulated forming a shell in space, for instance around a black hole horizon. It is an example where irregularities of the metric occur along a hypersurface, but the metric is regular enough, so that the curvature tensors make sense distributionally, see 
\cite{brady_stability_shell_around_schwarzschild_BH},
\cite{geroch_traschen_string_and_distributional_sources_in_GR},
\cite{israel_singular_hypersurfaces_and_thin_shells_in_GR}, \cite{taub_distribution_valued_curvature_tensors}, for related physics literature.

We now explain why Condition \eqref{H} can be interpreted as the presence of physical matter concentrated on the corner. As shown in \cite{lee_positivemasstheoremmanifolds_2015}, \cite{miao_cornerPMT}, this condition is the natural corner analogue of the Riemannian dominant energy condition. We observe that the same inequality arises from a more physical point of view, namely from the junction conditions of Israel in \cite{israel_singular_hypersurfaces_and_thin_shells_in_GR} for a cylindrical thin shell of matter, which can be viewed as a distributional analogue of the Einstein Equation for smooth spacetime metrics. We use the convention $[f]= f^- - f^+$ and the normal pointing from the $-$ side to the $+$ side. With this convention and rather simple computations including the spacetime second fundamental forms of the thin shell, Israel's second junction conditon gives
    \begin{equation*}
        H_{-}-H_{+}= 8\pi \sigma,
    \end{equation*}
    where $\sigma$ is the surface energy density. Thus, nonnegative surface energy density is consistent with Condition \eqref{H}, namely $H_{-} \geq H_{+}$.
\section{Smoothing, p-harmonic potentials and the Uniform $L^p$-Gradient Bound Along smoothed Corner}
The first problem one has to deal with when proving the Riemannian Penrose Inequality with corners via monotonicity formulas, is that the metric on the manifold is not smooth. We overcome this problem by Miao's smoothing argument across the corner.
\subsection{Miao's smoothing of corner metric}
Given $G = (g_{-}, g_{+})$ on $M$, we want to approximate $G$ by metrics which are $C^{2,\alpha}$ across $\Sigma$.
First, we use the Gaussian coordinates of $\Sigma$. Let $U^{2\epsilon}_-$ be a $2\epsilon$-tubular neighborhood of $\Sigma$ in $(\overline{\Omega}, g_{-})$ for some $\epsilon > 0$. We have the metric
\begin{equation}
 g^{-}_{ij}(x, t) \, dx^i dx^j + dt^2,
\end{equation}
on $\Sigma \times (-2\epsilon, 0]$. Where $t$ is the coordinate for $(-2\epsilon, 0]$, and $(x^1, \ldots, x^{n-1})$ are local coordinates for $\Sigma$ and $i, j \in [1, n - 1]$.
Similarly, for the $2\epsilon$-tubular neighborhood $U_+^{2\epsilon}$  of $\Sigma$ in $(M \setminus \Sigma, g_{+})$, we have the metric
\begin{equation}
    g^{+}_{ij}(x, t) \, dx^i dx^j + dt^2,
\end{equation}
where $t \in [0, 2\epsilon)$.
Since $g^{-}_{\Sigma} = g^{+}_{\Sigma}$, $G$ is a continuous metric $g$ on $M$. Inside $U^{2\epsilon} \cong \Sigma \times (-2\epsilon, 2\epsilon)$, we define
\begin{equation}\label{eq:corner_metric_gaussian_normal}
g = g_{ij}(x, t) \, dx^i dx^j + dt^2, 
\end{equation}
where $g_{ij}(x, t) = g^{-}_{ij}(x, t)$ when $t \leq 0$ and $g_{ij}(x, t) = g^{+}_{ij}(x, t)$ when $t \geq 0$.
\begin{proposition}[\cite{miao_cornerPMT}, Proposition 3.1]
\label{scalar_curvature_g_tau}
Let $G=(g_{-},g_{+})$ be a metric admitting corners along $\Sigma$. Then, there exists a family of globally $C^{2,\alpha}$ metrics $\{g_\varepsilon\}$, on $M$, such that $g_{\varepsilon}\to g$ uniformly $g_\varepsilon=g_{-}$ in $\Omega \setminus (\Sigma \times (-\varepsilon/2, 0])$ and $g_\varepsilon= g_{+}$ in $(M\setminus \Omega)\setminus (\Sigma \times [0, \varepsilon/2))$. Furthermore, the scalar curvature of $g_\varepsilon$ satisfies the following bounds.
\begin{equation}
    R^{\varepsilon}(x,t)= \begin{cases} O(1),\ \ &(x,t)\in \Sigma \times \{\frac{\varepsilon^2}{100}< |t| \leq \frac{\varepsilon}{2}\}\\ \\
    O(1)+ 2(H_-(x)-H_+(x))\left(\frac{100}{\varepsilon ^2}\phi(\frac{100t}{\varepsilon^2})\right),\ \ &(x,t)\in \Sigma \times [-\frac{\varepsilon^2}{100},\frac{\varepsilon^2}{100}]
\end{cases},
\end{equation}
where $\phi \in C_c^{\infty}\left(\left[-1,1\right]\right)$ is a standard mollifier in $\mathbb{R}$ such that,
\begin{equation}
    0\leq \phi \leq 1, \ \ \ \text{and}\ \ \ \int_{-1}^{1} \phi(s) ds = 1.
\end{equation}
\end{proposition}
The proof of this statement adapted to the needs of Theorem \ref{Riemannian_Penrose_Corners} can be found in \cite[Appendix A]{Miao_McCormik_Penrose_like_Inequality}.
\subsection{p-harmonic function and a uniform gradient bound} Having the smoothed metric at hand, we can introduce the family of solutions $v_\varepsilon \in C^{1,\beta}(M)\cap W^{1,p}(M)$ to the following problem
\begin{equation} \label{main_pde_p_laplace}
\begin{cases}
    \Delta_{g_\varepsilon}^p v=0  &\text{in } M \\    
    v=0  &\text{on } \partial M\\
    v\to 1  &\text{at } \infty
\end{cases},
\end{equation}
for any $p\in(1,3)$, where $ \Delta_{g_\varepsilon}^p(\cdot)=\operatorname{div}_{g_\varepsilon}(|\nabla \cdot|_{g_\varepsilon}^{p-2}\nabla \cdot)$ is the p-Laplacian operator of $\left(M, g_\varepsilon\right)$ and $g_\varepsilon$ is the metric described in Proposition \ref{scalar_curvature_g_tau}.
\begin{proposition}[Existence, Uniqueness and Regularity]
For every $\varepsilon > 0$ and every $1 < p < 3$ there exists a unique
solution
\[
v_{\varepsilon} \in C_{\mathrm{loc}}^{1,\beta}(M)
\cap W_{\mathrm{loc}}^{1,p}(M),
\qquad
0 \leq v_{\varepsilon} < 1,
\]
to Problem \eqref{main_pde_p_laplace}. Moreover $v_{\varepsilon}$ is proper,
$\partial M = \{v_{\varepsilon} = 0\}$, and
$|\nabla v_{\varepsilon}|_{g_\varepsilon} > 0$ on $\partial M$.
If $p \neq 2$, then $v_{\varepsilon}$ is $C^{3,\alpha}$ on
$M \setminus \operatorname{Crit}(v_{\varepsilon})$; if $p = 2$,
then $v_{\varepsilon}$ is $C^{3,\alpha}$ on $M$.
\end{proposition}
\begin{proof}
 The existence of $v_\varepsilon$ is ensured by the fact that the manifold $M$ equipped with the smoothed metric $g_\varepsilon$, being asymptotically flat, is p-nonparabolic, \cite{li_symmetricgreenfunctionscomplete_1987}. Uniqueness is then guaranteed by the maximum principle for p-harmonic functions. The (optimal) $C^{1,\beta}$-regularity of such solutions for $p\neq 2$ comes from the nowadays well known result \cite{dibenedetto_alphalocalregularityweak_1983}, and again by the maximum principle, $v_\varepsilon \in \left[0,1\right]$, \cite[Corollary 2.21]{lindqvist_noteslaplaceequation_2017}, \cite[Theorem 6.5]{heinonen_nonlinearpotentialtheorydegenerate_2018}.

Regarding the regularity, for $p\neq 2$ an equation as \eqref{main_pde_p_laplace} is becoming degenerate when $\nabla v_\varepsilon =0$, and thus classical elliptic theory does not apply. By the result in \cite{dibenedetto_alphalocalregularityweak_1983} we only have $C^{1,\beta}$ regularity. However, outside their critical set, i.e where $\left |\nabla v_\varepsilon \right|_{g_\varepsilon} \neq 0$, they have $C^{3,\alpha}$-regularity by the following discussion. We consider $w_\varepsilon$ any family of $p$-harmonic functions with respect to the metric $g_\varepsilon$. On the open set where $|\nabla w_\varepsilon|_{g_\varepsilon} > 0$, the $p$-Laplace
equation can be written in non-divergence form as a uniformly elliptic
quasilinear equation whose principal coefficients are
\[
a^{ij}
=
g_\varepsilon^{ij}
+
(p-2)
\frac{\nabla^i w_\varepsilon \nabla^j w_\varepsilon}
{|\nabla w_\varepsilon|_{g_\varepsilon}^2}.
\]
Since $g_\varepsilon \in C^{2,\alpha}$, Schauder theory, see for instance \cite[Theorem 6.17 and Part II]{gilbarg_ellipticpartialdifferentialequations_2015}, yields
$w_\varepsilon \in C^{3,\alpha}$ locally on
$\{|\nabla w_\varepsilon|_{g_\varepsilon} > 0\}$. Since the manifold is smooth, the compatibility of the coordinate charts provides $w_\varepsilon \in C^{3,\alpha}(M)$ for $\varepsilon>0$ fixed. This is sufficient regularity to employ the Bochner identity and Sard's theorem. The Bochner identity will be used only on the regular set.
The lack of a global Sard theorem for $p \neq 2$ will be dealt with
through the regularization scheme recalled in Subsection \ref{brief_discussion_in_case_critical_values}.
\end{proof}
In addition, we have that $\partial M$ is described by the level set $\left\{ v_\varepsilon =0\right\}$ and $v_\varepsilon : M \rightarrow \left[0,1\right)$ is proper.
A useful consequence of this is that the boundary datum is attained regularly, as Hopf's Lemma predicts a sign for the gradient on the boundary, (see \cite[Section 2, Paragraph 2.2]{benatti_minkowskiinequalitycompleteriemannian_2024}). Moreover, we have the following $L^p$ bound for the gradient of any $p$-harmonic function $w_\varepsilon$, arising from the $p$-Laplacian with respect to the metric $g_\varepsilon$ in the region $\Sigma \times \left(-\frac{\varepsilon}{2}, \frac{\varepsilon}{2}\right)$, due to the properties of the metric $g_\varepsilon$ given in Proposition \ref{scalar_curvature_g_tau}.
\begin{lemma}[$\varepsilon$-Uniform $L^p$-Gradient Bound] \label{p_uniform_gradient_bound}For the family $v_\varepsilon$, for $\varepsilon$ sufficiently small there exists a constant $K=K(p)$, such that
         \begin{equation}
        \int_{\Sigma \times(-\frac{\varepsilon}{2},\frac{\varepsilon}{2})} |\nabla v_\varepsilon|_{g_\varepsilon}^p d\mu_{g_\varepsilon} \leq K.
    \end{equation}
\end{lemma}
\begin{proof}[Proof of Lemma \ref{p_uniform_gradient_bound}]
For every $q \in \Sigma$, there exists a coordinate neighborhood $V(q)\subset M$ and an $R(q)>0$, $R\ll1$ such that the geodesic ball $B_{2R(q)}(q) \subseteq V(q)$. There exist $0<\lambda_q <\Lambda_q< \infty$, such that $\lambda_q |\xi|^2\leq g^{ij} \xi_i \xi_j \leq \Lambda_q |\xi|^2$, for every $\xi \in \mathbb{R}^n$ in $B_{2R(q)}(q)$. Since $g_\varepsilon$ is uniformly close to $g$, the metric $g_\varepsilon$ will satisfy the condition

\begin{equation}
\frac{\lambda_q}{2}|\xi|^2\leq g^{ij}_\varepsilon \xi_i \xi_j \leq 2\Lambda_q |\xi|^2,\ \ \text{for every $\xi \in \mathbb{R}^n$, in $B_{2R(q)}(q)$},
\end{equation}
for every $q\in \Sigma$. From now on we denote $R_i=R(q_i)$. By compactness of $\Sigma$ there exist $\{q_i\}_{i=1,...,N} \subset \Sigma$ such that
$\underset {i}{\bigcup} B_{R_i}(q_i) $ is an open cover of $\Sigma$. Therefore for $\varepsilon$ small enough we have
    \begin{equation}
        \Sigma \times(-\frac{\varepsilon}{2},\frac{\varepsilon}{2}) \subseteq \underset {i}{\bigcup} B_{R_i}(q_i) \subseteq \underset {i}{\bigcup} B_{2R_i}(q_i) .
    \end{equation}
 For $\varepsilon$  suffieciently small we have
  \begin{align}
        \int_{\Sigma \times(-\frac{\varepsilon}{2},\frac{\varepsilon}{2})} |\nabla v_\varepsilon|_{g_\varepsilon}^p d\mu_{\varepsilon}
        \leq \int_{\underset {i}{\bigcup} B_{R_i}(q_i)} |\nabla v_\varepsilon|_{g_\varepsilon}^p d\mu_{g_\varepsilon}
        \leq \sum_{i=1}^N\int_{B_{R_i}(q_i)} |\nabla v_\varepsilon|_{g_\varepsilon}^p d\mu,
\end{align}
Setting $\lambda_0= \min_{i} \frac{\lambda_{q_i}}{2}$, $\Lambda_0= \max_{i} 2{\Lambda_{q_i}}$ we can apply the Caccioppoli Inequality as stated in Appendix \ref{appB} with $r=R_i$ and $R=2R_i$ and using $0\leq v_\varepsilon\leq 1$, see in particular Remark \ref{rem:B-use-lemma23}, to obtain
\[
\begin{aligned}
\int_{\Sigma \times (-\frac{\varepsilon}{2},\frac{\varepsilon}{2})}
|\nabla v_\varepsilon|_{g_\varepsilon}^{p}\, {d\mu}_{g_\varepsilon}
&\leq
2^{p+1}(p-1)^{p-1}\Lambda_0^{p/2}\lambda_0^{-3/2}
\sum_{i=1}^{N}
\frac{
\left|B_{2R_i}(q_i)\setminus B_{R_i}(q_i)\right|_g
}{R_i^p} \\
&=: K,
\end{aligned}
\]
where $g$ the metric described in \eqref{eq:corner_metric_gaussian_normal} and $K$ independent of $\varepsilon$.
\end{proof}
\subsection{Introducing the quasi-monotone quantity}\label{section:Intro_monotone_quantity}
Let us provide some essential tools and introduce the quantity of interest. For every $ p \in \left(1,3\right)$ we introduce the p-capacity of $\partial M$ in $\left(M, g_\varepsilon\right)$
\begin{equation}
    \operatorname{Cap}_{\varepsilon,p}\left(\partial M\right)= \inf \left\{ \int_M \left|\nabla f\right|_{g_\varepsilon}^p  {d\mu}_{g_\varepsilon} : f \in C_c^{\infty}\left(M\right), f=1\ \text{on}\ \partial M \right\},
\end{equation}
which, associated to the level sets of $v_\varepsilon$ is given by
\begin{equation}\label{p_capacity}
    \operatorname{Cap}_{\varepsilon,p}\left( \partial M\right) = \int_M \left|\nabla v_\varepsilon\right|_{g_\varepsilon}^p {d\mu}_{g_\varepsilon} = \int_{\left\{v= t \right\}}\left| \nabla v_\varepsilon\right|_{g_\varepsilon}^{p-1} {d\sigma}_{g_\varepsilon},
\end{equation}
for every regular value $t$ of $v_\varepsilon$, for every $\varepsilon>0$ fixed, (see \cite[Section 2]{benatti_minkowskiinequalitycompleteriemannian_2024}). To make the notation simpler we set
\begin{equation}\label{p_small_capacity}
   c_{\varepsilon,p}=\left(\frac{\operatorname{Cap}_{\varepsilon,p}\left( \partial M\right)}{4\pi}\right)^{\frac{1}{p-1}}.
\end{equation}
\begin{lemma} \label{lemma:capacity_bounds_and_convergence}
For every fixed $1 < p < 3$, the quantities
$\operatorname{Cap}_{\varepsilon,p}(\partial M)$ are bounded above and below
by positive constants independent of $\varepsilon$. Moreover, as
$\varepsilon \to 0$, $\operatorname{Cap}_{\varepsilon,p}(\partial M)$
converges to the $p$-capacity of $\partial M$ with respect to the
metric $g$ given in \eqref{eq:corner_metric_gaussian_normal}.
\end{lemma}
\begin{proof}
Recall that we use the notation $\mu_{g_\varepsilon}=\mu$. Let
\[
E_{\varepsilon,p}(\varphi)
:=
\int_M |\nabla \varphi|_{g_\varepsilon}^p \, {d\mu}_{g_\varepsilon},
\qquad
E_{g,p}(\varphi)
:=
\int_M |\nabla \varphi|_g^p \, d\mu_g .
\]
Since the metrics $g_\varepsilon$ are uniformly equivalent to the corner
metric $g$ there exists a constant $C \geq 1$, independent of
$\varepsilon$, such that
\[
C^{-1} E_{g,p}(\varphi)
\leq
E_{\varepsilon,p}(\varphi)
\leq
C E_{g,p}(\varphi)
\]
for every admissible function $\varphi$. Taking the infimum over all
admissible $\varphi$ gives
\[
C^{-1}\operatorname{Cap}_{g,p}(\partial M)
\leq
\operatorname{Cap}_{\varepsilon,p}(\partial M)
\leq
C\operatorname{Cap}_{g,p}(\partial M).
\]
Since $\operatorname{Cap}_{g,p}(\partial M)$ is finite and positive for
$1<p<3$, this gives the desired uniform upper and lower bounds. It remains to prove convergence.
Let $\delta>0$, choose a smooth admissible function $\varphi_\delta$ such that
\[
E_{g,p}(\varphi_\delta)
\leq
\operatorname{Cap}_{g,p}(\partial M)+\delta .
\]

By the construction given in Proposition \ref{scalar_curvature_g_tau} the energy densities agree outside $\Sigma \times \left(-\frac{\varepsilon}{2}, \frac{\varepsilon}{2} \right)$, while the
contribution from $\Sigma \times \left(-\frac{\varepsilon}{2}, \frac{\varepsilon}{2} \right)$ tends to zero by the uniform equivalence
of $g_\varepsilon$ and $g$, therefore
\[
\limsup_{\varepsilon\to0}
\operatorname{Cap}_{\varepsilon,p}(\partial M)
\leq
\lim_{\varepsilon\to0} E_{\varepsilon,p}(\varphi_\delta)
=
E_{g,p}(\varphi_\delta)
\leq
\operatorname{Cap}_{g,p}(\partial M)+\delta .
\]
Letting $\delta\to0$ gives
\[
\limsup_{\varepsilon\to0}
\operatorname{Cap}_{\varepsilon,p}(\partial M)
\leq
\operatorname{Cap}_{g,p}(\partial M).
\]
For the reverse inequality, let $v_\varepsilon$ be the capacitary potential
associated with the metric $g_\varepsilon$. Then
\[
E_{\varepsilon,p}(v_\varepsilon)
=
\operatorname{Cap}_{\varepsilon,p}(\partial M).
\]
The uniform
energy comparison gives a uniform $W^{1,p}$ bound with respect to $g$.
Hence, up to a subsequence,
\[
v_\varepsilon \rightharpoonup v
\quad\text{weakly in } W^{1,p}.
\]
and by lower semicontinuity of the $p$-energy,
\[
E_{g,p}(v)
\leq
\liminf_{\varepsilon\to0} E_{\varepsilon,p}(v_\varepsilon).
\]
Moreover, the limit $v$ is admissible for the capacity problem with respect to $g$ by continuity of the trace operator and therefore
\[
\operatorname{Cap}_{g,p}(\partial M)
\leq
E_{g,p}(v)
\leq
\liminf_{\varepsilon\to0}
\operatorname{Cap}_{\varepsilon,p}(\partial M).
\]
\end{proof}
\begin{remark}
    In the setting where $M$ has a boundary, the estimate of Lemma \ref{p_uniform_gradient_bound} could also be obtained from the uniform energy bound in Lemma \ref{lemma:capacity_bounds_and_convergence}. However, we keep the local proof since the same argument applies in the punctured setting away from the Green's function pole, where a global capacity estimate does not immediately provide the desired bound.
\end{remark}
Following \cite{agostiniani_riemannianpenroseinequalitynonlinear_2022}, we consider the following vector field
\begin{equation} \label{p_vector_field}
Y_{\varepsilon,p}=\frac{c_{\varepsilon,p}^{\frac{p-1}{3-p}}}{\left[\frac{3-p}{p-1}\left(1-v_{\varepsilon}\right)\right]^{\frac{p-1}{3-p}}} \left \{\frac{\left|\nabla v_{\varepsilon} \right|_{g_\varepsilon}^{p-2}\nabla v_{\varepsilon}}{c_{\varepsilon,p}^{p-1}} + \frac{\nabla\left|\nabla v_{\varepsilon} \right|_{g_\varepsilon}- \frac{\Delta v_{\varepsilon}}{\left|\nabla v_{\varepsilon}\right|_{g_\varepsilon}}\nabla v_{\varepsilon}}{\frac{3-p}{p-1}\left(1-v_{\varepsilon}\right)} + \frac{\left|\nabla v_{\varepsilon}\right|_{g_\varepsilon} \nabla v_{\varepsilon}}{\left[ \frac{3-p}{p-1}\left(1-v_{\varepsilon}\right)\right]^2}\right \},
\end{equation}
where all the differential operators are with respect to the metric $g_\varepsilon$. 
We introduce the function
\begin{equation} \label{p_quantity_in_flux_form}
    Q_{\varepsilon,p}(t)= \int_{\left\{ v_{\varepsilon} = \zeta_p(t) \right\}} \left \langle Y_{\varepsilon,p}, \frac{\nabla v_{\varepsilon}}{\left| \nabla v_{\varepsilon} \right|_{g_\varepsilon}}\right \rangle_{g_\varepsilon} {d\sigma}_{g_\varepsilon},
\end{equation}
where 
\begin{equation}\label{p_level_sets}
    \zeta_p(t)=1-\left(\frac{t_p}{t}\right)^{\frac{3-p}{p-1}}, \ \ \ \text{with}\ \ \ t_p=\left(\frac{p-1}{3-p}c_{\varepsilon,p}\right)^{\frac{p-1}{3-p}}, \ \ \ t\in\left[t_p, +\infty\right).
\end{equation}
Where in the notation we have dropped the explicit dependence on $\varepsilon$ for its simplicity. By expanding the equation $\Delta_{g_\varepsilon}^pv_{\varepsilon}=0$ away from the critical points we obtain
\begin{equation} \label{normal_laplacian_of_p_harmonic}
    \Delta_{g_\varepsilon}v_{\varepsilon}=\left(2-p\right) \frac{\left\langle \nabla\left|\nabla v_{\varepsilon}\right|_{g_\varepsilon}, \nabla v_{\varepsilon}\right \rangle_{g_\varepsilon}}{\left |\nabla v_{\varepsilon}\right|_{g_\varepsilon}},
\end{equation}
where all the operators are with respect to the metric $g_\varepsilon$ and the Laplacian is written with the index $g_\varepsilon$ so that we remain consistent with all the notion above. The mean curvature of the level set $\left\{v_{\varepsilon}= \tau\right\}$ computed with respect to the infinity pointing unit normal $\nu:=\nu\left(\varepsilon,p\right)= \frac{\nabla v_{\varepsilon}}{\left|\nabla v_{\varepsilon} \right|_{g_\varepsilon}}$ can be written as
\begin{equation} \label{p_mean_curvature}
  H:=H^\varepsilon(p)=\frac{\Delta v_{\varepsilon}}{\left|\nabla v_{\varepsilon}\right|_{g_\varepsilon}}-\frac{\left\langle \nabla\left|\nabla v_{\varepsilon} \right|_{g_\varepsilon}, \nabla v_{\varepsilon} \right \rangle_{g_\varepsilon}}{\left |\nabla v_{\varepsilon} \right|_{g_\varepsilon}^2}=-\left(p-1\right)\frac{\left\langle \nabla\left|\nabla v_{\varepsilon} \right|_{g_\varepsilon}, \nabla v_{\varepsilon} \right \rangle_{g_\varepsilon}}{\left |\nabla v_{\varepsilon} \right|_{g_\varepsilon}^2}
\end{equation}
With the Expressions \eqref{p_capacity}, \eqref{p_small_capacity} and \eqref{p_mean_curvature}, one can easily see that the quantity given by \eqref{p_quantity_in_flux_form} can be expressed as
\begin{equation} \label{p_quantity}
     Q_{\varepsilon,p}(t)= 4\pi t- \frac{t^{\frac{2}{p-1}}}{c_{\varepsilon,p}}\int_{\left \{v_{\varepsilon} = \zeta_p (t)\right\}} \left|\nabla v_{\varepsilon} \right|_{g_\varepsilon} H {d\sigma}_{g_\varepsilon} +\frac{t^{\frac{5-p}{p-1}}}{c_{\varepsilon,p}^2} +\int_{\left \{v_{\varepsilon} = \zeta_p (t)\right\}} \left|\nabla v_{\varepsilon} \right|_{g_\varepsilon}^2 {d\sigma}_{g_\varepsilon},
\end{equation}

\section{Approximate Monotonicity Formulas} In this section we will firstly treat the case where the set $\operatorname{Crit}v_{\varepsilon}:= \left\{ x\in M : \left| \nabla v_{\varepsilon}\right|_{g_\varepsilon}=0\right \}$ is empty, and then we will briefly discuss the way to obtain the general statement, (see Subsection \ref{brief_discussion_in_case_critical_values}). With all the information provided above, we are ready to state and prove the following proposition, which is the argument in \cite[Subsection 1.2]{agostiniani_riemannianpenroseinequalitynonlinear_2022} adapted to a manifold equipped with Miao's smoothed metric.
\begin{proposition}\label{p_lemma_firstgoodbound}
     Let $(M, g_\varepsilon)$ be a 3-dimensional, complete, non-compact, asymptotically flat Riemannian manifold (admitting a corner as $\Sigma$ described before) with compact and connected boundary $\partial M$, where $g_\varepsilon$ the metric of Proposition \ref{scalar_curvature_g_tau}. Let $v_{\varepsilon}$ be the family of solutions to \eqref{main_pde_p_laplace}. Then, there exist $0< s_\varepsilon \leq t_\varepsilon < +\infty$ with $\Sigma \times(-\frac{\varepsilon}{2},\frac{\varepsilon}{2})\subset\left\{\zeta_p\left(s_\varepsilon\right)<v_\varepsilon< \zeta_p\left(t_\varepsilon\right)\right\}$ such that \begin{equation}
        Q_{\varepsilon,p}\left(t_\varepsilon\right)-Q_{\varepsilon,p}\left(s_\varepsilon\right)\geq \frac{1}{2} \frac{c_{\varepsilon,p}^{\frac{p-1}{3-p}}}{\left[ \frac{3-p}{p-1}\right]^{\frac{p-1}{3-p}+1}}\int_{\left\{\zeta_p\left(s_\varepsilon\right)< v_\varepsilon<\zeta_p\left(t_\varepsilon\right)\right\}}\frac{R^{\varepsilon}|\nabla v_{\varepsilon}|_{g_\varepsilon}}{(1-v_{\varepsilon})^{\frac{p-1}{3-p}+1}} {d\mu}_{g_\varepsilon}.
    \end{equation}
where $p \in\left(1,3\right)$,   $\zeta_p\left(s_\varepsilon\right)$  and $\zeta_p\left(t_\varepsilon\right)$  correspond to regular values of $v_{\varepsilon}$ and $R^\varepsilon$ is the scalar curvature of the metric $g_\varepsilon$.
 \end{proposition}
The scalar curvature term is due to the construction of Miao, which does not provide a metric with nonnegative scalar curvature, but only bounded below, Proposition \ref{scalar_curvature_g_tau}. We will see later that the right hand side of Inequality \eqref{p_first_good_bound} can be bounded below by a quantity that becomes zero as $\varepsilon \rightarrow 0$.
 
\begin{proof}[Proof of Proposition \ref{p_lemma_firstgoodbound}]
    Using the Bochner formula, the twice contracted Gauss Equation and \eqref{normal_laplacian_of_p_harmonic}, \eqref{p_mean_curvature}, we can compute the divergence of the vector field $Y_{\varepsilon,p}$.
\begin{align}\label{p_divergence_ready}
   \operatorname{div}\left(Y_{\varepsilon,p}\right)= \frac{c_{\varepsilon,p}^{\frac{p-1}{3-p}}\left|\nabla v_{\varepsilon}\right|_{g_\varepsilon}}{\left[ \frac{3-p}{p-1}\left(1-v_{\varepsilon}\right)\right]^{\frac{p-1}{3-p}+1}} \Bigg\{ \frac{\left|\nabla v_{\varepsilon} \right|_{g_\varepsilon}^{p-1}}{c_{\varepsilon,p}^{p-1}}- \frac{R^\varepsilon_{\Sigma_\varepsilon}}{2}+&\frac{\left|\nabla^{\Sigma_\varepsilon}\left|\nabla v_{\varepsilon}\right|_{g_\varepsilon}\right|_{g_\varepsilon}^2}{\left| \nabla v_{\varepsilon} \right|_{g_\varepsilon}^2}+ \frac{R^\varepsilon}{2}
   +\frac{|\overset{\circ}{h^\varepsilon}|_{g_\varepsilon}^2}{2}+\\
        &+ \frac{5-p}{p-1}\left(\frac{|\nabla v_\varepsilon|_{g_\varepsilon}}{\frac{3-p}{p-1}\left(1-v_{\varepsilon}\right)}-H \right)^2\Bigg\},
        \end{align}
    where $\Sigma_\varepsilon= \left\{v_{\varepsilon}=v_{\varepsilon}(x)\right\}$ is the level set passing through $x\in M$ and $R^\varepsilon_{\Sigma_\varepsilon}$, $\nabla^{\Sigma_\varepsilon}$, $\overset{\circ}{h^\varepsilon}$ are the scalar curvature, the Levi-Civita connection and the traceless second fundamental form of the level set $\Sigma_\varepsilon$. Clearly, all the quantities and differential operators are with respect to the metric $g_\varepsilon$. Let
\[
a_{\varepsilon}:=
\min_{\Sigma\times[-\varepsilon/2,\varepsilon/2]} v_{\varepsilon},
\qquad
b_{\varepsilon}:=
\max_{\Sigma\times[-\varepsilon/2,\varepsilon/2]} v_{\varepsilon}.
\]

Choose regular values \(a_{\varepsilon}^{-}<a_{\varepsilon}\) and
\(b_{\varepsilon}^{+}>b_{\varepsilon}\) sufficiently close to
\(a_{\varepsilon}\) and \(b_{\varepsilon}\), respectively, and define
\(s_{\varepsilon},t_{\varepsilon}\) by
\(\zeta_p(s_{\varepsilon})=a_{\varepsilon}^{-}\) and
\(\zeta_p(t_{\varepsilon})=b_{\varepsilon}^{+}\). Then
\[
\Sigma\times(-\varepsilon/2,\varepsilon/2)
\subset
\{\zeta_p(s_{\varepsilon})<v_{\varepsilon}<\zeta_p(t_{\varepsilon})\}.
\] as in Figure \ref{fig:my-picture}. Using the Divergence Theorem and the Coarea Formula as in \cite{agostiniani_riemannianpenroseinequalitynonlinear_2022} one obtains
    \begin{align}
Q_{\varepsilon,p}\left(t_\varepsilon\right)-Q_{\varepsilon,p}\left(s_\varepsilon\right)&= \int_{\left\{v_\varepsilon=\zeta_p\left(t_\varepsilon\right) \right\}} \left \langle Y_{\varepsilon,p}, \frac{\nabla v_\varepsilon}{\left| \nabla v_\varepsilon \right|_{g_\varepsilon}}\right \rangle {d\sigma}_{g_\varepsilon}- \int_{\left\{v_\varepsilon=\zeta_p\left(s_\varepsilon\right) \right\}} \left \langle Y_{\varepsilon,p}, \frac{\nabla v_\varepsilon}{\left| \nabla v_\varepsilon \right|_{g_\varepsilon}}\right \rangle {d\sigma}_{g_\varepsilon}\\
        &=\int_{\left\{\zeta_p\left(s_\varepsilon\right)<v_\varepsilon<\zeta_p\left(t_\varepsilon\right)\right\}} \operatorname{div}Y_{\varepsilon,p} {d\mu}_{g_\varepsilon}= \int_{\left\{\zeta_p\left(s_\varepsilon\right), \zeta_p\left(t_\varepsilon\right) \right\}}d\tau \int_{\{v_\varepsilon=\tau \}} \frac{\operatorname{div}\left(Y_{\varepsilon,p}\right)}{\left|\nabla v_\varepsilon \right|_{g_\varepsilon}} {d\sigma}_{g_\varepsilon}.
    \end{align}
    By substituting Expression \eqref{p_divergence_ready} in the above we obtain
    \begin{align}
        Q_{\varepsilon,p}\left(t_\varepsilon\right)-Q_{\varepsilon,p}\left(s_\varepsilon\right)
     \geq& \int_{\left\{\zeta_p\left(s_\varepsilon\right), \zeta_p\left(t_\varepsilon\right) \right\}} \frac{c_{\varepsilon,p}^{\frac{p-1}{3-p}}}{\left[ \frac{3-p}{p-1}\left(1-\tau\right)\right]^{\frac{p-1}{3-p}+1}} d\tau  \int_{\{v_\varepsilon=\tau \}}  \left\{ \frac{\left|\nabla v_\varepsilon \right|_{g_\varepsilon}^{p-1}}{c_{\varepsilon,p}^{p-1}}- \frac{R^\varepsilon_{\Sigma_\varepsilon}}{2}\right\} {d\sigma}_{g_\varepsilon}\\
     &+  \int_{\left\{\zeta_p\left(s_\varepsilon\right), \zeta_p\left(t_\varepsilon\right) \right\}} \frac{c_{\varepsilon,p}^{\frac{p-1}{3-p}}}{\left[ \frac{3-p}{p-1}\left(1-\tau\right)\right]^{\frac{p-1}{3-p}+1}}d\tau  \int_{\{v_\varepsilon=\tau \}}  \frac{R^\varepsilon}{2} {d\sigma}_{g_\varepsilon}\\
     \geq&  \int_{\left\{\zeta_p\left(s_\varepsilon\right), \zeta_p\left(t_\varepsilon\right) \right\}}  \frac{c_{\varepsilon,p}^{\frac{p-1}{3-p}}\left[4\pi-2\pi \chi \left(\left\{v_\varepsilon=\tau\right\}\right) \right]}{\left[ \frac{3-p}{p-1}\left(1-\tau\right)\right]^{\frac{p-1}{3-p}+1}} d\tau\\ &+  \int_{\left\{\zeta_p\left(s_\varepsilon\right), \zeta_p\left(t_\varepsilon\right) \right\}} \frac{c_{\varepsilon,p}^{\frac{p-1}{3-p}}}{\left[ \frac{3-p}{p-1}\left(1-\tau\right)\right]^{\frac{p-1}{3-p}+1}}d\tau  \int_{\{v_\varepsilon=\tau \}}  \frac{R^\varepsilon}{2} {d\sigma}_{g_\varepsilon}\\
     \geq&\int_{\left\{\zeta_p\left(s_\varepsilon\right), \zeta_p\left(t_\varepsilon\right) \right\}} \frac{c_{\varepsilon,p}^{\frac{p-1}{3-p}}}{\left[ \frac{3-p}{p-1}\left(1-\tau\right)\right]^{\frac{p-1}{3-p}+1}}d\tau  \int_{\{v_\varepsilon=\tau \}}  \frac{R^\varepsilon}{2} {d\sigma}_{g_\varepsilon},
    \end{align}
    where we absorbed the manifestly nonnegative terms and used the identities \eqref{p_capacity}, \eqref{p_small_capacity} in combination with the Gauss-Bonnet Theorem, obtaining the Euler characteristic $\chi \left(\left\{v_\varepsilon=\tau\right\}\right)$ of the level set $\left\{v_\varepsilon=\tau\right\}$. Since we are considering absence of critical points in this part, all the level sets are closed and connected, since they are diffeomorphic to the boundary $\partial M$. As $\partial M$ is described by the smooth connected and closed level set $\{v_{\varepsilon}= 0\}$, we have that $4\pi-2\pi \chi \left(\left\{v_\varepsilon=\tau\right\}\right)\geq 0$. We have thus obtained the inequality
    \begin{equation}
Q_{\varepsilon,p}\left(t_\varepsilon\right)-Q_{\varepsilon,p}\left(s_\varepsilon\right)\geq \int_{\left\{\zeta_p\left(s_\varepsilon\right), \zeta_p\left(t_\varepsilon\right) \right\}} \frac{c_{\varepsilon,p}^{\frac{p-1}{3-p}}}{\left[ \frac{3-p}{p-1}\left(1-\tau\right)\right]^{\frac{p-1}{3-p}+1}}d\tau  \int_{\{v_\varepsilon=\tau \}}  \frac{R^\varepsilon}{2} {d\sigma}_{g_\varepsilon}.
    \end{equation}
Using the Coarea Formula again we conclude
    \begin{equation}
        Q_{\varepsilon,p}\left(t_\varepsilon\right)-Q_{\varepsilon,p}\left(s_\varepsilon\right)\geq \frac{1}{2} \frac{c_{\varepsilon,p}^{\frac{p-1}{3-p}}}{\left[ \frac{3-p}{p-1}\right]^{\frac{p-1}{3-p}+1}}\int_{\left\{\zeta_p\left(s_\varepsilon\right)< v_\varepsilon<\zeta_p\left(t_\varepsilon\right)\right\}}\frac{R^{\varepsilon}|\nabla v_{\varepsilon}|_{g_\varepsilon}}{(1-v_\varepsilon)^{\frac{p-1}{3-p}+1}} {d\mu}_{g_\varepsilon}.
    \end{equation}
\end{proof}
\subsection{Approximate monotonicity in presence of critical values/points}\label{brief_discussion_in_case_critical_values} The at most $C^{1,\beta}$-regularity for $p\neq 2$ becomes a problem in the analysis, when one has to treat the monotonicity along non regular values.
This issue has been extensively discussed and resolved in \cite{agostiniani_riemannianpenroseinequalitynonlinear_2022}, where the authors noticed that such regularity does not suffice to invoke Sard's Theorem in order to argue that the set of critical values is negligible.  In particular, the goal would be to obtain an expression of the type
 \begin{equation}
Q_{\varepsilon,p}\left(t_\varepsilon\right)-Q_{\varepsilon,p}\left(s_\varepsilon\right)\geq \int_{\left\{\zeta_p\left(s_\varepsilon\right), \zeta_p\left(t_\varepsilon\right) \right\}\setminus \mathcal{N}_\varepsilon} \frac{c_{\varepsilon,p}^{\frac{p-1}{3-p}}}{\left[ \frac{3-p}{p-1}\left(1-\tau\right)\right]^{\frac{p-1}{3-p}+1}}d\tau  \int_{\{v=\tau \}}  \frac{R^\varepsilon}{2} {d\sigma}_{g_\varepsilon},
    \end{equation}
    where $\mathcal{N}_\varepsilon= v_\varepsilon\left(\operatorname{Crit}\left(v_\varepsilon\right)\right)$. One can deal with this issue by using an approximation scheme, originally introduced in \cite{dibenedetto_alphalocalregularityweak_1983}. Another issue arising from the presence of critical values is that of the connectedness of the level sets. To overcome this, we add the topological assumption $H_2(M;\mathbb{Z})=\{0\}$ in the general statement, \cite{huisken_inversemeancurvatureflow_2001}. For the sake of completeness, we briefly discuss this scheme in this section.\\

The idea is to approximate locally the solution $v_\varepsilon$ of Problem \eqref{main_pde_p_laplace} with a family of $C^{3,\alpha}$ functions that solve a perturbed version of that problem. Fix $T\in\left(0,1\right)$ such that $\left\{v_\varepsilon= T\right\}$ is a regular level set of $v_\varepsilon$. For every $\delta >0$, we consider the unique solution $v_\varepsilon^\delta$ to the problem
\begin{equation}\label{p_snooth_solutions}
    \begin{cases}
        \operatorname{div}_{g_\varepsilon}
        \left(|\nabla v_\varepsilon^\delta|^{p-2}_{\delta,g_\varepsilon}\nabla v_\varepsilon^\delta\right)=0 &\text{in $M_T=\left\{0\leq v_\varepsilon \leq T\right\}$},\\
        v_\varepsilon^\delta=0 &\text{on $\partial M$},\\
        v_\varepsilon^\delta=T &\text{on $\left\{v_\varepsilon=T\right\}$},
    \end{cases}
\end{equation}
where all the differential operators are with respect to the metric $g_\varepsilon$ and $|\nabla v_\varepsilon|_{\delta,g_\varepsilon}= \sqrt{|\nabla v_\varepsilon|_{g_\varepsilon}^2 + \delta^2}$. These solutions converge in $C^{1,\beta}$ topology to the $W^{1,p}$ solutions of Problem \eqref{main_pde_p_laplace} in compact subsets of $M_T$ as $\delta\rightarrow 0$. For this section we again drop the index $p$ from the solution. The dependencies on $\varepsilon$ remain the same as in the previous section, while the dependencies on $\delta$ will be explicitly denoted. In analogy to \eqref{p_vector_field}, we define the corresponding quantity associated to $v_\varepsilon^\delta$.
\begin{equation}\label{p_delta_vector_field}
    Y_{\varepsilon,p}^{\delta} =\frac{c_{\varepsilon,\delta}^{\frac{p-1}{3-p}}}{\left[\frac{3-p}{p-1}\left(1-v_\varepsilon^\delta\right)\right]^{\frac{p-1}{3-p}}} \left \{\frac{\left|\nabla v_\varepsilon^\delta \right|_{\delta,g_\varepsilon}^{p-2}\nabla v_\varepsilon^\delta}{c_{\varepsilon,\delta}^{p-1}} + \frac{\nabla\left|\nabla v_\varepsilon^\delta \right|_{g_\varepsilon}- \frac{\Delta v_\varepsilon^\delta}{\left|\nabla v_\varepsilon^\delta\right|_{g_\varepsilon}}\nabla v_\varepsilon^\delta}{\frac{3-p}{p-1}\left(1-v_\varepsilon^\delta\right)} + \frac{\left|\nabla v_\varepsilon^\delta\right|_{g_\varepsilon} \nabla v_\varepsilon^\delta}{\left[ \frac{3-p}{p-1}\left(1-v_\varepsilon^\delta\right)\right]^2}\right \},
\end{equation}
where \begin{equation}
    c_{\varepsilon,\delta}^{p-1}= \frac{1}{4\pi} \int_{\partial M} \left|\nabla v_\varepsilon^\delta\right|_{\delta,g_\varepsilon} \left|\nabla v_\varepsilon^\delta \right|_{g_\varepsilon} {d\sigma}_{g_\varepsilon}= \frac{1}{4\pi} \int_{\left
    \{v_\varepsilon^\delta=\zeta_{p}^\delta(t)\right\}}
    \left|\nabla v_\varepsilon^\delta\right|_{\delta,g_\varepsilon}^{p-2} \left|\nabla v_\varepsilon^\delta\right|_{g_\varepsilon} {d\sigma}_{g_\varepsilon},
\end{equation}
where the second equality holds for every $t$ for which $\zeta_{p}^{\delta}(t)$ is a regular value of $v_\varepsilon^\delta$, by the Divergenece Theorem. In analogy to \eqref{p_quantity_in_flux_form}, we introduce
\begin{equation}\label{p_di_benedetto_flux_form}
    Q_{\varepsilon,p}^{\delta}(t)= \int_{\left\{v_\varepsilon^\delta= \zeta_{p}^\delta(t) \right\}}\left\langle Y_{\varepsilon,p}^{\delta}, \frac{\nabla v_\varepsilon^\delta}{\left|\nabla v_\varepsilon^\delta\right|_{g_\varepsilon}}\right\rangle {d\sigma}_{g_\varepsilon},
\end{equation}
where \begin{equation}\label{form_of_level_sets}
    \zeta_{p}^{\delta}= 1- \left(\frac{t_p^{\delta}}{t}\right)^{\frac{3-p}{p-1}}, \ \ \ \text{with}\ \ \ t_p^{\delta}=\left(c_{\varepsilon,\delta} \frac{p-1}{3-p}\right)^{\frac{p-1}{3-p}},
\end{equation}
for $t\in \left[t_p^{\delta}, t_p^{\delta}\left(1-T\right)^{\frac{1-p}{3-p}}\right]$. The function $Q_{\varepsilon,p}^{\delta}$ is well defined whenever $\zeta_{p}^\delta(t)$ is a regular value of $v_\varepsilon^\delta$. Away from the critical set, the equation $\operatorname{div}\left( \left|\nabla v_\varepsilon^\delta\right|_{\delta,g_\varepsilon}^{p-2}\nabla v_\varepsilon^\delta\right)=0$ is equivalent to 
\begin{equation}
    \Delta v_\varepsilon^\delta= \left(2-p\right) \frac{\left|\nabla v_\varepsilon^\delta\right|_{g_\varepsilon}^2}{\left|\nabla v_\varepsilon^\delta\right|_{\delta,g_\varepsilon}^2}\frac{\left\langle \nabla\left|\nabla v_\varepsilon^\delta\right|_{g_\varepsilon}, \nabla v_\varepsilon^\delta\right \rangle}{\left|\nabla v_\varepsilon^\delta\right|_{g_\varepsilon}},
\end{equation}
and as consequence, the mean curvature of a regular level set of $v_\varepsilon^\delta$ with respect to the infinity-pointing unit normal $\nu=\frac{\nabla v_\varepsilon^\delta}{\left|\nabla v_\varepsilon^\delta\right|_{g_\varepsilon}}$ can be expressed as
\begin{equation}
    H_{\varepsilon}^{\delta}= \frac{\Delta v_\varepsilon^\delta}{\left|\nabla v_\varepsilon^\delta\right|_{g_\varepsilon}}-\frac{\left\langle \nabla\left|\nabla v_\varepsilon^\delta\right|_{g_\varepsilon},\nabla v_\varepsilon^\delta\right \rangle}{\left|\nabla v_\varepsilon^\delta \right|_{g_\varepsilon}}=-\frac{\left(p-1\right)\left|\nabla v_\varepsilon^\delta\right|_{g_\varepsilon}^2+\delta^2}{\left|\nabla v_\varepsilon^\delta\right|_{g_\varepsilon}^2+\delta^2}\frac{\left\langle \nabla\left| \nabla v_\varepsilon^\delta\right|_{g_\varepsilon}, \nabla v_\varepsilon^\delta\right \rangle}{\left|\nabla v_\varepsilon^\delta \right|_{g_\varepsilon}^2}.
\end{equation}
With the above at hand, we can write in correspondence to \eqref{p_quantity} the following quantity.
\begin{equation}
Q_{\varepsilon,p}^{\delta}(t)= 4\pi t - \frac{t^{\frac{2}{p-1}}}{c_{\varepsilon}^\delta} \int_{\left\{v_\varepsilon^\delta= \zeta_{p}^{\delta}(t) \right\}} \left| \nabla v_\varepsilon \right|_{g_\varepsilon} H_\varepsilon^{\delta} {d\sigma}_{g_\varepsilon} + \frac{t^{\frac{5-p}{p-1}}}{c_{\varepsilon,\delta}^2} \int_{\left\{v_\varepsilon= \zeta_{p}^{\delta}(t) \right\}} \left|\nabla v_\varepsilon^\delta\right|_{g_\varepsilon}^2 {d\sigma}_{g_\varepsilon}.
\end{equation}
The statement we can prove in presence of critical values using this quantity is the following.
\begin{proposition}\label{p_first_good_bound_critical_values_in}
 Let $(M, g_\varepsilon)$ be a 3-dimensional, complete, non-compact, asymptotically flat Riemannian manifold (admitting a corner as $\Sigma$ described before) with compact and connected boundary $\partial M$, satisfying $H_2(M,\partial M;\mathbb{Z})=\{0\}$ where $g_\varepsilon$ the metric of Proposition \ref{scalar_curvature_g_tau}. Let $v_\varepsilon^\delta$ be the family of solutions to \eqref{p_snooth_solutions}.  Then, there exist $t_p^\delta< s_\varepsilon \leq t_\varepsilon < t_p^{\delta}\left(1-T\right)^{\frac{1-p}{3-p}}$ with $\Sigma \times(-\frac{\varepsilon}{2},\frac{\varepsilon}{2})\subset\left\{\zeta_{p}^{\delta}\left(s_\varepsilon\right)<v_\varepsilon^\delta< \zeta_{p}^{\delta}\left(t_\varepsilon\right)\right\}$ such that $Q_{\varepsilon,p}^{\delta}$ satisfies \begin{align}
    Q_{\varepsilon,p}^{\delta}(t)- Q_{\varepsilon,p}^{\delta}(s) \geq- \delta \left(\frac{p+1}{p-1}\right)^2 \int_{\left\{\zeta_{p}^\delta(s) <v_\varepsilon^\delta < \zeta_{p}^\delta(t)\right\}\setminus \mathcal{N}_\varepsilon} \frac{\delta \left|\nabla v_\varepsilon\right|_{g_\varepsilon}}{2\left(p+1\right)\left|\nabla v_\varepsilon \right|_{g_\varepsilon}^2+ 3\delta^2}\frac{c_{\varepsilon,\delta}^{\frac{p-1}{3-p}}\left|\nabla v_\varepsilon\right|_{g_\varepsilon}^2}{\left[ \frac{3-p}{p-1}\left(1-v_\varepsilon\right) \right]^{\frac{p-1}{3-p}+3}} {d\mu}_{g_\varepsilon}.
\end{align}
for $0<s\leq t$ in either $\left[t_p^\delta, s_\varepsilon\right]$ or $ \left[t_\varepsilon, t_p^{\delta}\left(1-T\right)^{\frac{1-p}{3-p}}\right]$, and \begin{align}
\label{p_first_good_bound} \\
    Q_{\varepsilon,p}^{\delta}(t_\varepsilon) &- Q_{\varepsilon,p}^{\delta}(s_\varepsilon) \\
    &\geq \frac{c_{\varepsilon,\delta}^{\frac{p-1}{3-p}}}{2\left(\frac{3-p}{p-1}\right)^{\frac{p-1}{3-p}+1}}  \int_{\Sigma \times(-\frac{\varepsilon}{2},\frac{\varepsilon}{2})\setminus \operatorname{Crit}(v_\varepsilon)} \frac{R^{\varepsilon}|\nabla v_{\varepsilon}|_{g_\varepsilon}}{(1-v_\varepsilon)^{\frac{p-1}{3-p}+1}} {d\mu}_{g_\varepsilon}\\  & - \delta \left(\frac{p+1}{p-1}\right)^2 \int_{\left\{\zeta_{p}^{\delta}(s_\varepsilon) <v_\varepsilon^\delta < \zeta_{p}^{\delta}(t_\varepsilon)\right\}\setminus \mathcal{N}_\varepsilon}\frac{\delta \left|\nabla v_\varepsilon\right|_{g_\varepsilon}}{2\left(p+1\right)\left|\nabla v_\varepsilon \right|_{g_\varepsilon}^2+ 3\delta^2}\frac{c_{\varepsilon,\delta}^{\frac{p-1}{3-p}}\left|\nabla v_\varepsilon\right|_{g_\varepsilon}^2}{\left[ \frac{3-p}{p-1}\left(1-v_\varepsilon^\delta\right) \right]^{\frac{p-1}{3-p}+3}} {d\mu}_{g_\varepsilon},
\end{align}
where $\operatorname{Crit}(v_\varepsilon)$ and $\mathcal{N}_\varepsilon$ the sets of critical points and critical values of $v_\varepsilon$, respectively, $p \in\left(1,3\right)$,   $\zeta_p\left(s_\varepsilon\right)$  and $\zeta_p\left(t_\varepsilon\right)$  correspond to regular values of $v_\varepsilon$ and $R^\varepsilon$ is the scalar curvature of the metric $g_\varepsilon$.
\end{proposition}
\begin{proof}
    The proof of the above statement is a direct combination of the proof of Proposition \ref{p_lemma_firstgoodbound} and \cite[Proof of Lemma 1.2]{agostiniani_riemannianpenroseinequalitynonlinear_2022}. If $|\nabla v_\varepsilon|_{g_\varepsilon}\neq 0$, computing the divergence of the vector field \eqref{p_delta_vector_field} and using the divergence theorem yields the thesis. To treat the general case we define the vector field
    \begin{align}
Z_{\varepsilon,p}^\delta=&\frac{{c_{p,\varepsilon}^{\delta}}^{\frac{p-1}{3-p}}}{\left[\frac{3-p}{p-1}\left(1-v_\varepsilon \right)\right]^{\frac{p-1}{3-p}}}\Bigg\{ \frac{\nabla|\nabla v_\varepsilon|_{g_\varepsilon}-\frac{\Delta v_\varepsilon}{|\nabla v_\varepsilon|_{g_\varepsilon}}\nabla v_\varepsilon}{\frac{3-p}{p-1}(1-v_\varepsilon)}+\frac{|\nabla v_\varepsilon|_{g_\varepsilon}\nabla v_\varepsilon}{\left[\frac{3-p}{p-1}(1-v_\varepsilon)\right]^2}\Bigg\}\\
=&\frac{{c_{p,\varepsilon}^{\delta}}^{\frac{p-1}{3-p}}}{\left[\frac{3-p}{p-1}\left(1-v_\varepsilon\right)\right]^{\frac{p-1}{3-p}}}\Bigg\{\frac{\nabla^{\top}|\nabla v_\varepsilon|_{g_\varepsilon}+(p-1)\nabla^{\perp}|\nabla v_\varepsilon|_{g_\varepsilon}}{\frac{3-p}{p-1}(1-v_\varepsilon)}+\frac{|\nabla v_\varepsilon|_{g_\varepsilon}^2}{\left[\frac{3-p}{p-1}(1-v_\varepsilon) \right]^2}\frac{\nabla v_\varepsilon}{|\nabla v|_{g_\varepsilon}}
+\frac{(2-p)\delta^2}{|\nabla v_\varepsilon|_{g_\varepsilon}^2+\delta^2}\frac{\nabla^{\perp}|\nabla v_\varepsilon|_{g_\varepsilon}}{\frac{3-p}{p-1}(1-v_\varepsilon)}\Bigg\}.
    \end{align}
    Then the vector fields $Y_{\varepsilon,p}^\delta$ and $Z_{\varepsilon,p}^\delta$ are related through the formula
    \begin{equation}
        Y_{\varepsilon,p}^\delta={c_{\varepsilon,p}^\delta}^{-\frac{(2-p)(p-1)}{3-p}}\frac{|\nabla v_\varepsilon|_{\delta,g_\varepsilon}^{p-2}\nabla v_\varepsilon}{\left[ \frac{3-p}{p-1}(1-v_\varepsilon)\right]^{\frac{p-1}{3-p}}}+Z_{\varepsilon,p}^{\delta}.
    \end{equation}
    The above make sense only outside the critical set. So we use the cut-off functions $\eta_k:\left[ 0,+\infty\right)\rightarrow \left[0,1\right]$ such that
    \begin{equation}
        \eta_k(\tau)\equiv 0\ \ \text{in $\left[0,\frac{1}{2k}\right]$},\ \ \ 0\leq\eta'_k(\tau)\leq 2k\ \ \text{in $\left[\frac{1}{2k},\frac{3}{2k}\right]$,}\ \ \ \eta_k(\tau)\equiv1\ \ \text{in $\left[\frac{3}{2k},+\infty\right)$},
    \end{equation}
    for every $k\in \mathbb{N}$ and we consider the vector fields
    \begin{equation}
        Z_{\varepsilon,p,k}^\delta= \eta_k\left(\frac{(p-1)|\nabla v_\varepsilon|_{g_\varepsilon}}{\left[\frac{3-p}{p-1}(1-v_\varepsilon)\right]^{\frac{1}{3-p}}}\right)Z_{\varepsilon,p}^\delta\ \ \text{and}\ \ Y_{\varepsilon,p,k}^\delta={c_{p,\varepsilon}^{\delta}}^{-\frac{(2-p)(p-1)}{3-p}} \frac{|\nabla v_\varepsilon|^{p-2}_{\delta,g_\varepsilon}\nabla v_\varepsilon}{\left[\frac{3-p}{p-1}(1-v_\varepsilon)\right]^{\frac{p-1}{3-p}}}+ Y_{\varepsilon,p}^\delta,
    \end{equation}
    which are well defined and smooth on the whole $M_T$. By computing the divergence of these vector fields, using the divergence theorem by exploiting Equation \eqref{p_di_benedetto_flux_form}, applying appropriate manipulations and using the coarea formula, \cite[Proof of Lemma 1.2]{agostiniani_riemannianpenroseinequalitynonlinear_2022} one arrives at
    \begin{equation}
        Q_{\varepsilon,p}^\delta(t)-Q_{\varepsilon,p}^\delta(s)\geq \int_{\left\{\zeta_{p}^{\delta}(s)< v_\varepsilon< \zeta_{p}^{\delta}(t)\right\}} \left\{ {c_{p,\varepsilon}^{\delta}}^{-\frac{(2-p)(p-1)}{3-p}} \frac{|\nabla v_\varepsilon|^{p-2}_{\delta,g_\varepsilon}\nabla v_\varepsilon}{\left[\frac{3-p}{p-1}(1-v_\varepsilon)\right]^{\frac{p-1}{3-p}+1}} +\mathbb{I}_{M_T \setminus \operatorname{Crit}(v_\varepsilon)}\operatorname{div}{Z_{\varepsilon,p}^\delta}\right\} {d\mu}_{g_\varepsilon},
    \end{equation}
    where $\mathbb{I}_{M_T \setminus \operatorname{Crit}(v_\varepsilon)}$ denotes the characteristic function of $M_T \setminus \operatorname{Crit}(v_\varepsilon)$. By the use of Sard's Theorem and of the coarea formula one finally obtains
    \begin{equation}
         Q_{\varepsilon,p}^\delta(t)-Q_{\varepsilon,p}^\delta(s)\geq \int_{(\zeta_{p}^{\delta}(s),\zeta_{p}^\delta(t))\setminus \mathcal{N}_\varepsilon} d\tau \int_{\{v_\varepsilon=\tau\}} \frac{\operatorname{div}Y_{\varepsilon,p}}{|\nabla v_\varepsilon|_{g_\varepsilon}} {d\sigma}_{g_\varepsilon}.
    \end{equation}The argument for the connectedness of the level sets in presence of critical values can be found in \cite[Subsection 1.3]{agostiniani_riemannianpenroseinequalitynonlinear_2022}, where the assumption $H_2(M,\partial M;\mathbb{Z})=\{0\}$ is used. After this point, the proof is the same as the proof of Proposition \ref{p_lemma_firstgoodbound}.
\end{proof}

\subsection{Estimating the remainder in the quasi monotonicity} In this section we prove an approximate monotonicity for Quantity \eqref{p_quantity}.
\begin{theorem}[Approximate monotonicity in absence of critical values]\label{p_approximate_monotonicity_lemma}
      Let $(M, g_\varepsilon)$ be a 3-dimensional, complete, non-compact, asymptotically flat Riemannian manifold (admitting a corner as $\Sigma$ described before) with compact and connected boundary $\partial M$, where $g_\varepsilon$ the metric of Proposition \ref{scalar_curvature_g_tau}. Let $v_\varepsilon$ the family of solutions to \eqref{main_pde_p_laplace}. Then, there exist $0< s_\varepsilon \leq t_\varepsilon < +\infty$ with $\Sigma \times(-\frac{\varepsilon}{2},\frac{\varepsilon}{2})\subset\left\{\zeta_p\left(s_\varepsilon\right)<v_\varepsilon< \zeta_p\left(t_\varepsilon\right)\right\}$, $\zeta_p(s_\varepsilon)$, $\zeta_p(t_\varepsilon)$ are regular values of $v_\varepsilon$ such that $Q_{\varepsilon,p}$ is non-decreasing in $\left[t_p, s_\varepsilon\right] \cup \left[t_\varepsilon, +\infty\right)$ and a constant $C=C(P,g,\Sigma,\partial M)$ such that
      \begin{equation}
         Q_{\varepsilon,p}(t_\varepsilon)- Q_{\varepsilon,p}(s_\varepsilon) \geq -C\varepsilon^{\frac{p-1}{p}}.
      \end{equation}
    
 \end{theorem}
\begin{proof}[Proof of Theorem \ref{p_approximate_monotonicity_lemma}]
 By the construction of the metric $g_\varepsilon$ we know that in ${\left\{\zeta_p\left(s_\varepsilon\right)< v_\varepsilon<\zeta_p\left(t_\varepsilon\right)\right\}} \setminus \Sigma \times(-\frac{\varepsilon}{2},\frac{\varepsilon}{2})$ the scalar curvature is nonnegative. Therefore, by this information in combination with Proposition \ref{p_lemma_firstgoodbound}
 \begin{align}\label{first_good_bound_regular}
        Q_{\varepsilon,p}\left(t_\varepsilon\right)-Q_{\varepsilon,p}\left(s_\varepsilon\right)\geq& \frac{1}{2} \frac{c_{\varepsilon,p}^{\frac{p-1}{3-p}}}{\left[ \frac{3-p}{p-1}\right]^{\frac{p-1}{3-p}+1}}\int_{\Sigma \times(-\frac{\varepsilon}{2},\frac{\varepsilon}{2})}\frac{R^{\varepsilon}|\nabla v_{\varepsilon}|_{g_\varepsilon}}{(1-v_\varepsilon)^{\frac{p-1}{3-p}+1}} {d\mu}_{g_\varepsilon}\\ \geq&- \frac{1}{2} \frac{c_{\varepsilon,p}^{\frac{p-1}{3-p}}}{\left[ \frac{3-p}{p-1}\right]^{\frac{p-1}{3-p}+1}}\int_{\Sigma \times(-\frac{\varepsilon}{2},\frac{\varepsilon}{2})}\frac{R^{\varepsilon}_-|\nabla v_{\varepsilon}|_{g_\varepsilon}}{(1-v)^{\frac{p-1}{3-p}+1}} {d\mu}_{g_\varepsilon},
    \end{align}
   where by $R^{\varepsilon}_-$ we denote the negative part of the scalar curvature in the smoothing region. Since the smoothing region is
contained in a fixed compact set, there exists $T_0<1$, independent of small
$\varepsilon$, such that
\[
v_{\varepsilon} \leq T_0
\qquad \text{on } \Sigma \times (-\varepsilon/2,\varepsilon/2).
\]
Then
\[
(1-v_{\varepsilon})^{-\left(\frac{p-1}{3-p}+1\right)}
\leq
(1-T_0)^{-\left(\frac{p-1}{3-p}+1\right)} .
\]
This, together with the uniform boundedness of the capacity yields
\[
Q_{\varepsilon,p}(t_\varepsilon)-Q_{\varepsilon,p}(s_\varepsilon)
\geq
-C\int_{\Sigma\times(-\frac{\varepsilon}{2},\frac{\varepsilon}{2})}
R_\varepsilon^- |\nabla v_\varepsilon|_{g_\varepsilon}\, {d\mu}_{g_\varepsilon} .
\]
By Proposition \ref{scalar_curvature_g_tau} combined with Condition \eqref{H}, H\"older Inequality, Lemma \ref{p_uniform_gradient_bound} and the fact that
$\operatorname{Vol}_{g_\varepsilon}\bigl(\Sigma \times
(-\varepsilon/2,\varepsilon/2)\bigr)=O(\varepsilon)$, one obtains
\[
\begin{aligned}
Q_{\varepsilon,p}(t_\varepsilon)-Q_{\varepsilon,p}(s_\varepsilon)
&\geq
-C \int_{\Sigma\times(-\frac{\varepsilon}{2},\frac{\varepsilon}{2})}
|\nabla v_{\varepsilon}|_{g_\varepsilon}\, d\mu_{g_\varepsilon} \\
&\geq
-C \operatorname{Vol}_{g_\varepsilon}\bigl(\Sigma \times
(-\varepsilon/2,\varepsilon/2)\bigr)^{(p-1)/p}
\left(
\int_{\Sigma\times(-\frac{\varepsilon}{2},\frac{\varepsilon}{2})}
|\nabla v_{\varepsilon}|_{g_\varepsilon}^p\, d\mu_{g_\varepsilon}
\right)^{1/p} \\
&\geq
-C\varepsilon^{(p-1)/p}.
\end{aligned}
\] 
\end{proof}
The corresponding general statement for Theorem \ref{p_approximate_monotonicity_lemma}, in which we take critical values under consideration, is the following.
\begin{theorem}[Approximate monotonicity with critical values] \label{p_app_app_monotonicity}Let $(M, g_\varepsilon)$ be a 3-dimensional, complete, non-compact, asymptotically flat Riemannian manifold (admitting a corner as $\Sigma$ described before) with compact and connected boundary $\partial M$ satisfying $H_2(M,\partial M;\mathbb{Z})=\{0\}$, where $g_\varepsilon$ the metric of Proposition \ref{scalar_curvature_g_tau}. Let $v_\varepsilon^\delta$ the family of solutions to \eqref{p_snooth_solutions}. Then, there exist $0< s_\varepsilon \leq t_\varepsilon < +\infty$ with $\Sigma \times(-\frac{\varepsilon}{2},\frac{\varepsilon}{2})\subset\left\{\zeta_{p}^\delta\left(s_\varepsilon\right)<v_\varepsilon^\delta< \zeta_{p}^\delta\left(t_\varepsilon\right)\right\}$ such that $Q_{\varepsilon,p}^{\delta}$ satisfies
      \begin{equation}
         Q_{\varepsilon,p}^{\delta}(t)- Q_{\varepsilon,p}^{\delta}(s) \geq - \delta \left(\frac{p+1}{p-1}\right)^2 \int_{\left\{\zeta_{p}^\delta(s) <v_\varepsilon^\delta < \zeta_{p}^\delta(t)\right\}} \frac{\delta \left|\nabla v_\varepsilon^\delta\right|_{g_\varepsilon}}{2\left(p+1\right)\left|\nabla v_\varepsilon^\delta \right|_{g_\varepsilon}^2+ 3\delta^2}\frac{{c_{\varepsilon,p}^\delta}^{\frac{p-1}{3-p}}\left|\nabla v_\varepsilon^\delta\right|_{g_\varepsilon}^2}{\left[\frac{3-p}{p-1}\left(1-v_\varepsilon^\delta\right) \right]^{\frac{p-1}{3-p}+3}} {d\mu}_{g_\varepsilon}
      \end{equation}
       for $0<s\leq t$ in either $\left[t_p^\delta, s_\varepsilon\right]$ or $ \left[t_\varepsilon, t_p^{\delta}\left(1-T\right)^{\frac{1-p}{3-p}}\right]$ and a constant $C_\delta=C_\delta(p,g,\Sigma, \partial M)$, such that
           \begin{align} Q_{\varepsilon,p}^{\delta}(t_\varepsilon)- Q_{\varepsilon,p}^{\delta}(s_\varepsilon) \geq& -C_\delta \varepsilon^{\frac{p-1}{p}}\\ &- \delta \left(\frac{p+1}{p-1}\right)^2 \int_{\left\{\zeta_{p}^\delta(s_\varepsilon) <v_\varepsilon^\delta < \zeta_{p}^\delta(t_\varepsilon)\right\}} \frac{\delta \left|\nabla v_\varepsilon^\delta\right|_{g_\varepsilon}}{2\left(p+1\right)\left|\nabla v_\varepsilon^\delta \right|_{g_\varepsilon}^2+ 3\delta^2}\frac{{c_{\varepsilon,p}^\delta}^{\frac{p-1}{3-p}}\left|\nabla v_\varepsilon^\delta\right|_{g_\varepsilon}^2}{\left[ \frac{3-p}{p-1}\left(1-v_\varepsilon^\delta\right) \right]^{\frac{p-1}{3-p}+3}} {d\mu}_{g_\varepsilon},
      \end{align}
      where $\zeta_p^{\delta}(s_\varepsilon)$, $\zeta_p^{\delta}(t_\varepsilon)$ are regular values of $v_\varepsilon$.
 \end{theorem}
\begin{proof}[Proof of Theorem \ref{p_app_app_monotonicity}]
The proof of this statement is a direct combination of the proof of Proposition \ref{p_first_good_bound_critical_values_in} and Theorem \ref{p_approximate_monotonicity_lemma}.
\end{proof}
\begin{remark}\label{lemma_of_constant_delta}
$
    \lim_{\delta\rightarrow 0}C_\delta=C$,
    where $C$ the constant appearing in Theorem \ref{p_approximate_monotonicity_lemma}, by the fact that the solutions of \eqref{p_snooth_solutions} converge in $C^{1,\beta}$ topology to the $W^{1,p}$ solutions of \eqref{main_pde_p_laplace} in compact subsets of $M_T$.
\end{remark}
\section{Positive Mass Theorem and Riemannian Penrose Inequality with Corners}
In this section we provide the proofs of Theorems \ref{MIAO} and \ref{Riemannian_Penrose_Corners}. Both manifolds described in Theorems \ref{MIAO} and \ref{Riemannian_Penrose_Corners} satisfy the topological assumption by the discussions in \cite[Proposition 2.1]{bray_harmonicfunctionsmassdimensional_2022} and \cite[Lemma 4.1]{huisken_inversemeancurvatureflow_2001}, respectively.

\subsection{Proof of the Riemannian Penrose Inequality with Corners} We are now ready to use the approximate monotonicity we obtained in Theorem \ref{p_app_app_monotonicity} to prove Theorem \ref{Riemannian_Penrose_Corners}. By the second assumption in Theorem \ref{Riemannian_Penrose_Corners} and \cite[Lemma 4.1]{huisken_inversemeancurvatureflow_2001}, we have that $M$ is diffeomorphic to $\mathbb{R}^3\setminus B^3$ and $\partial M$ is diffeomorphic to $\mathbb{S}^2$ and thus we have that $H_2(M,\partial M;\mathbb{Z})=\{0\}$. As a result, we can invoke Theorem \ref{p_app_app_monotonicity}.
\begin{proof}[Proof of Theorem \ref{Riemannian_Penrose_Corners}]Let $1 < p \leq 2$ be fixed. We have
\[
Q_{\varepsilon,p}^{\delta}(t_{p}^{\delta})
=
4\pi t_{p}^{\delta}
+
\frac{(t_p^{\delta})^{(5-p)/(p-1)}}{(c_{\varepsilon,p}^{\delta})^2}
\int_{\partial M}
|\nabla v_{\varepsilon}^{\delta}|_{g_\varepsilon}^2\, {d\sigma}_{g_\varepsilon}
\geq
4\pi t_{p}^{\delta},
\]
because $H_{\partial M}=0$. From the approximate monotonicity, Theorem \ref{p_app_app_monotonicity}, after letting $\delta \to 0$ we get
\[
4\pi t_{p}
\leq Q_{\varepsilon,p}(t_\varepsilon)+C\varepsilon^{(p-1)/p},
\]
where
\[
4\pi t_{p}
=
(4\pi)^{\frac{2-p}{3-p}}
\left(\frac{p-1}{3-p}\right)^{\frac{p-1}{3-p}}
\operatorname{Cap}_{\varepsilon,p}(\partial M)^{\frac{1}{3-p}},
\]
 which can be obtained by a simple computation using Equations \eqref{p_small_capacity} and \eqref{p_level_sets}.
By the fact that $Q_{\varepsilon,p}$ is monotone on $[t_\varepsilon,+\infty)$ and by
Appendix \ref{AppA},
\[
Q_{\varepsilon,p}(t_\varepsilon)
\leq \limsup_{t\to+\infty} Q_{\varepsilon,p}(t)
\leq 8\pi m_{\mathrm{ADM}}(M,G).
\]
Therefore, we obtain
\[
(4\pi)^{\frac{2-p}{3-p}}
\left(\frac{p-1}{3-p}\right)^{\frac{p-1}{3-p}}
\operatorname{Cap}_{\varepsilon,p}(\partial M)^{\frac{1}{3-p}}
\leq
8\pi m_{\mathrm{ADM}}(M,G)+C\varepsilon^{(p-1)/p}.
\]
Letting $\varepsilon\to 0$ and using the convergence of the $p$-capacities yields
\[
(4\pi)^{\frac{2-p}{3-p}}
\left(\frac{p-1}{3-p}\right)^{\frac{p-1}{3-p}}
\operatorname{Cap}_{G,p}(\partial M)^{\frac{1}{3-p}}
\leq
8\pi m_{\mathrm{ADM}}(M,G).
\]
By the discussion done in \cite[Lemma 2.4]{agostiniani_riemannianpenroseinequalitynonlinear_2022} and references therein, we know that
\[
\lim_{p\to 1^+}\operatorname{Cap}_{G,p}(\partial M)=|\partial M|.
\]
Finally, letting $p\to 1^+$ we obtain

\begin{equation}
    m_{ADM}(M,G) \geq \sqrt{\frac{\left|\partial M \right|}{16 \pi}}.
\end{equation}
\end{proof}
\subsection{Proof of the Positive Mass Theorem with Corners}\label{pmtcornerproof} In this section we briefly guide the reader into how to obtain the proof of Theorem \ref{MIAO}.  \begin{proposition}
Let $u_\varepsilon$ be the Green's function potential solving
\[
\Delta_{g_\varepsilon}u_\varepsilon = 4\pi\delta_o,
\qquad
u_\varepsilon \to 1 \quad \text{at infinity,}
\]
with $o\notin \Sigma\times\left(-\frac{\varepsilon}{2},\frac{\varepsilon}{2}\right)$.
Assume $H_2(M;\mathbb{Z})=\{0\}$. Then the associated Green's function
quantity $Q^o_{\varepsilon,2}$ satisfies the same approximate monotonicity
across the smoothing region, namely
\[
Q^o_{\varepsilon,2}(t_\varepsilon)
-
Q^o_{\varepsilon,2}(s_\varepsilon)
\geq
-C\varepsilon^{1/2},
\]
with $C$ independent of $\varepsilon$.
\end{proposition}

\begin{proof}[Proof of Theorem \ref{MIAO}]
The asymptotic behavior near the pole, see \cite{agostiniani_greenfunctionproofpositive_2024} gives
\[
\lim_{s\to 0} Q^o_{\varepsilon,2}(s)=0.
\]
The sharp asymptotic estimate at infinity, Appendix \ref{AppA} gives
\[
\limsup_{t\to+\infty} Q^o_{\varepsilon,2}(t)
\leq 8\pi m_{\mathrm{ADM}}(M,G).
\]
Using the approximate monotonicity through the corner, we obtain
\[
8\pi m_{\mathrm{ADM}}(M,G)\geq -C\varepsilon^{1/2}.
\]
Letting $\varepsilon\to 0$ yields $m_{\mathrm{ADM}}(M,G)\geq 0$.
\end{proof}

\section*{Appendices}
\appendices
\section{Asymptotically Flat Manifolds, ADM Mass and Sharp Asymptotic Control}
\label{AppA}

For completeness, we recall the notion of asymptotic flatness used throughout the paper and the
normalization of the ADM mass. We then record the sharp asymptotic estimate for the quantity
$Q_{\varepsilon,p}$ introduced in Section \ref{section:Intro_monotone_quantity}. This estimate replaces the separate appendix on
optimal decay assumptions.

\subsection{Asymptotically flat manifolds and ADM mass}

\begin{definition}[$C^1_\tau$-asymptotic flatness]\label{def:A-af}
A complete three-dimensional Riemannian manifold $(M,g)$, with possibly empty compact boundary
and one end, is said to be $C^1_\tau$-asymptotically flat, $\tau>0$, if there exists a compact set
$K\subset M$ such that $M\setminus K$ is diffeomorphic to $\mathbb R^3\setminus B_R$ through a
coordinate chart $x=(x^1,x^2,x^3)$, and, in this chart,
\[
 g=g_{ij}\,dx^i\otimes dx^j=(\delta_{ij}+h_{ij})\,dx^i\otimes dx^j,
\]
with
\[
\sum_{i,j=1}^3\sum_{|\beta|\leq 1}
 |x|^{\tau+|\beta|}\,|\partial^\beta h_{ij}|=O(1),
 \qquad |x|\to+\infty .
\]
\end{definition}

For such a manifold, the ADM mass is defined, in an asymptotically flat coordinate chart, by
\[
 m_{\rm ADM}(M,g)
=\lim_{r\to+\infty}\frac1{16\pi}
 \int_{\{|x|=r\}}
(\partial_j g_{ij}-\partial_i g_{jj})\frac{x^i}{|x|}\,d\sigma_{\mathbb R^3},
\]
whenever the limit exists.  If $\tau>1/2$, this flux is a geometric invariant, namely it does not
depend on the chosen asymptotically flat coordinate chart, provided the limit is finite.  In the
nonnegative scalar curvature setting one may also regard the ADM mass as an extended value; it is
finite under the usual scalar-curvature integrability assumption.

\subsection{Sharp asymptotic estimate for the $p$-Hawking quantity}

We now prove the sharp estimate needed in the proof of the Positive Mass Theorem and of the
Riemannian Penrose Inequality with corners. The key point is that, after passing to the logarithmic
variable, the quantity $Q_{\varepsilon,p}$ is exactly $8\pi$ times the $p$-Hawking mass with the
normalization used in \cite{benatti_nonlinearisocapacitaryconceptsmass_2023}. The factor $8\pi$ is the only normalization constant entering the final estimate.

Throughout this subsection, we fix $1<p\leq2$ and $\varepsilon>0$. Let $v_\varepsilon$ be the solution
of Problem \eqref{main_pde_p_laplace} on $(M,g_\varepsilon)$ and set
\[
 w_\varepsilon=-(p-1)\log(1-v_\varepsilon).
\]
Then $w_\varepsilon$ solves
\begin{equation}\label{eq:A-w-problem}
\begin{cases}
\Delta^{(p)}_{g_\varepsilon}w_\varepsilon=|\nabla w_\varepsilon|_{g_\varepsilon}^p & \text{in } M,\\
w_\varepsilon=0 & \text{on } \partial M,\\
w_\varepsilon\to +\infty & \text{at infinity},
\end{cases}
\end{equation}
where
\[
 \Delta^{(p)}_{g_\varepsilon}f
 =\operatorname{div}_{g_\varepsilon}(|\nabla f|_{g_\varepsilon}^{p-2}\nabla f).
\]
For a regular value $s$ of $w_\varepsilon$, we write
\[
\Omega_s=\{w_\varepsilon\leq s\},
\qquad
\Sigma_s=\partial\Omega_s=\{w_\varepsilon=s\},
\]
and we compute the mean curvature $H$ of $\Sigma_s$ with respect to the normal pointing toward
infinity. We also use the $p$-capacity normalization of \cite{benatti_nonlinearisocapacitaryconceptsmass_2023}, namely
\[
c'_{\varepsilon,p}(\partial\Omega_s)
=\frac1{4\pi}\int_{\Sigma_s}
\left(\frac{|\nabla w_\varepsilon|_{g_\varepsilon}}{3-p}\right)^{p-1}\,d\sigma
=e^s c'_{\varepsilon,p}(\partial M).
\]
With the notation of Section \ref{section:Intro_monotone_quantity}, one has
\[
t_p=\big(c'_{\varepsilon,p}(\partial M)\big)^{\frac1{3-p}},
        \qquad
\{v_\varepsilon=\zeta_p(t)\}
=\{w_\varepsilon=(3-p)\log(t/t_p)\}.
\]
Thus it is convenient to set
\[
F_{\varepsilon,p}(s)
:=Q_{\varepsilon,p}\left(t_p e^{\frac{s}{3-p}}\right).
\]
Expanding $Q_{\varepsilon,p}$ in the variable $s$ gives
\[
\begin{aligned}
F_{\varepsilon,p}(s)
&=\big(c'_{\varepsilon,p}(\partial M)\big)^{\frac1{3-p}}e^{\frac{s}{3-p}}
\left[4\pi
 -\int_{\Sigma_s}\frac{|\nabla w_\varepsilon|_{g_\varepsilon}}{3-p}H\,d\sigma
 +\int_{\Sigma_s}\frac{|\nabla w_\varepsilon|_{g_\varepsilon}^2}{(3-p)^2}\,d\sigma\right].
\end{aligned}
\]
On the other hand, the $p$-Hawking mass of \cite[Formula (2.12)]{benatti_nonlinearisocapacitaryconceptsmass_2023} is
\[
\begin{aligned}
\mathfrak m_{H,\varepsilon}^{(p)}(\Sigma_s)
=\frac{\big(c'_{\varepsilon,p}(\partial\Omega_s)\big)^{\frac1{3-p}}}{8\pi}
\left[4\pi
+\int_{\Sigma_s}\frac{|\nabla w_\varepsilon|_{g_\varepsilon}^2}{(3-p)^2}\,d\sigma
 -\int_{\Sigma_s}\frac{|\nabla w_\varepsilon|_{g_\varepsilon}}{3-p}H\,d\sigma\right].
\end{aligned}
\]
Since $c'_{\varepsilon,p}(\partial\Omega_s)=e^s c'_{\varepsilon,p}(\partial M)$, we obtain the exact identity
\begin{equation}\label{eq:A-F-equals-8pi-mH}
 F_{\varepsilon,p}(s)
 =Q_{\varepsilon,p}\left(t_p e^{\frac{s}{3-p}}\right)
 =8\pi\,\mathfrak m_{H,\varepsilon}^{(p)}(\Sigma_s).
\end{equation}
This identity is the place where the factor $8\pi$ enters.

\begin{lemma}[Sharp asymptotic control of the $p$-Hawking mass]\label{lem:A-sharp-phawking}
Let $(M,G)$ satisfy the assumptions of Theorem \ref{Riemannian_Penrose_Corners}, and let $g_\varepsilon$ be the metric given by
Proposition \ref{scalar_curvature_g_tau}. For every fixed $\varepsilon>0$ and every $1<p\leq2$, the level sets of
$w_\varepsilon$ satisfy
\begin{equation}\label{eq:A-normalized-sharp}
\limsup_{s\to+\infty}\mathfrak m_{H,\varepsilon}^{(p)}(\Sigma_s)
 \leq m_{\rm ADM}(M,g_\varepsilon)
= m_{\rm ADM}(M,G).
\end{equation}
\end{lemma}

\begin{proof}
Let
\[
\Sigma\times\left(-\frac\varepsilon2,\frac\varepsilon2\right)
\]
be the region where the smoothing of the corner takes place.  Since $w_\varepsilon$ is proper and
$\mathcal A_\varepsilon$ is compact, there exists $s_\varepsilon>0$ such that
$\mathcal A_\varepsilon\subset\Omega_{s_\varepsilon}$.  Hence, for $s\geq s_\varepsilon$, the whole exterior
region $M\setminus\Omega_s$ is contained in the part of the manifold where $g_\varepsilon$ coincides
with the original exterior metric.  In particular, on $M\setminus\Omega_s$ the scalar curvature is
nonnegative, the metric is $C^1_\tau$-asymptotically flat, $\tau>1/2$, and the ADM mass is
$m_{\rm ADM}(M,g_\varepsilon)=m_{\rm ADM}(M,G)$.

We now use the asymptotic comparison of $p$-Hawking masses in \cite[Proposition 4.1]{benatti_nonlinearisocapacitaryconceptsmass_2023}. The proof
of that result applies to the tail of the present level-set flow, because all sufficiently large levels
lie in the region where $R_{g_\varepsilon}\geq0$ and where the metric is the original asymptotically flat
exterior metric. Choosing $q=2$ gives
\[
 \limsup_{s\to+\infty}\mathfrak m_{H,\varepsilon}^{(p)}(\Sigma_s)
\leq
\limsup_{s\to+\infty}\mathfrak m_{H,\varepsilon}^{(2)}(\Sigma_s).
\]
Here, when $q=2$, $\mathfrak m_{H,\varepsilon}^{(2)}(\Sigma_s)$ denotes the quasi-local
$2$-Hawking mass of the set $\Omega_s$, computed with the $2$-capacitary potential in the
exterior region $M\setminus\Omega_s$.  For every sufficiently large regular level $s$, this exterior
region has nonnegative scalar curvature and ADM mass $m_{\rm ADM}(M,g_\varepsilon)$.  The
optimal-decay estimate for the $2$-Hawking mass in \cite[Theorem 4.11]{benatti_isoperimetricriemannianpenroseinequality_2022}, applied to these exterior
regions and then to the limit superior as $s\to+\infty$, gives
\[
\limsup_{s\to+\infty}\mathfrak m_{H,\varepsilon}^{(2)}(\Sigma_s)
 \leq m_{\rm ADM}(M,g_\varepsilon).
\]
Combining the previous two inequalities yields
\[
\limsup_{s\to+\infty}\mathfrak m_{H,\varepsilon}^{(p)}(\Sigma_s)
 \leq m_{\rm ADM}(M,g_\varepsilon).
\]
Finally, $g_\varepsilon$ agrees with the original exterior metric outside a compact set, hence its ADM
mass is exactly the ADM mass of the corner metric $G$.
\end{proof}

\begin{proposition}[Sharp form of the asymptotic estimate]\label{prop:A-sharp-Q}
Under the assumptions above, the quantity $Q_{\varepsilon,p}$ satisfies
\begin{equation}\label{eq:A-final-sharp}
\limsup_{\varepsilon\to0}\,\limsup_{t\to+\infty} Q_{\varepsilon,p}(t)
 \leq 8\pi\,m_{\rm ADM}(M,G).
\end{equation}
Equivalently, in the normalization of \cite{benatti_nonlinearisocapacitaryconceptsmass_2023},
\[
\limsup_{\varepsilon\to0}\,\limsup_{s\to+\infty}
 \mathfrak m_{H,\varepsilon}^{(p)}(\Sigma_s)\leq m_{\rm ADM}(M,G).
\]
\end{proposition}

\begin{proof}
By the change of variables $t=t_p e^{s/(3-p)}$, letting $t\to+\infty$ is the same as letting
$s\to+\infty$. Combining \eqref{eq:A-F-equals-8pi-mH} with Lemma~\ref{lem:A-sharp-phawking},
for every fixed $\varepsilon>0$ we get
\[
\begin{aligned}
\limsup_{t\to+\infty}Q_{\varepsilon,p}(t)
&=\limsup_{s\to+\infty}F_{\varepsilon,p}(s) \\
&=8\pi\limsup_{s\to+\infty}\mathfrak m_{H,\varepsilon}^{(p)}(\Sigma_s) \\
&\leq 8\pi\,m_{\rm ADM}(M,g_\varepsilon)
=8\pi\,m_{\rm ADM}(M,G).
\end{aligned}
\]
Taking the upper limit as $\varepsilon\to0$ gives \eqref{eq:A-final-sharp}.
\end{proof}
\section{Caccioppoli Inequality for Charts} \label{appB}

We record a local Caccioppoli inequality for weakly $p$-harmonic functions in the form used
in Lemma \ref{p_uniform_gradient_bound}. The estimate itself is intrinsic.  The coordinate chart only enters when one
chooses cut-off functions and, later on, when the Riemannian volume of a fixed coordinate
annulus is compared with a background volume.

\begin{proposition}[Caccioppoli inequality for the $p$-Laplace--Beltrami operator in charts]\label{prop:B-caccioppoli}
Let $(M,g)$ be a Riemannian manifold, let $1<p<\infty$, and let
$u\in W^{1,p}_{\rm loc}(U)$ be a weak solution of
\[
        \Delta_p^g u:=\operatorname{div}_g\big(|\nabla u|_g^{p-2}\nabla u\big)=0
\]
in an open set $U\subset M$.  Let $V(q)$ be a coordinate chart and let
$\Omega\Subset U\cap V(q)$.  Assume that, in this chart,
\[
        g^{ij}(x)\xi_i\xi_j\leq \Lambda |\xi|^2,
        \qquad x\in\Omega,\quad \xi\in\mathbb R^n,
\]
for some $\Lambda<\infty$.  If $B_r\subset B_R\Subset\Omega$ are concentric coordinate balls,
then
\[
        \int_{B_r}|\nabla u|_g^p\,d\mu_g
        \leq
        \frac{2^{p+1}(p-1)^{p-1}\Lambda^{p/2}}{(R-r)^p}
        \int_{B_R\setminus B_r}|u|^p\,d\mu_g .
\]
In particular, if in addition
\[
        \lambda |\xi|^2\leq g^{ij}(x)\xi_i\xi_j\leq \Lambda |\xi|^2
        \qquad\hbox{on }\Omega,
\]
then, for $0\leq |u|\leq1$,
\[
        \int_{B_r}|\nabla u|_g^p\,d\mu_g
        \leq
        \frac{2^{p+1}(p-1)^{p-1}\Lambda^{p/2}}{(R-r)^p}
        \mu_g(B_R\setminus B_r)
        \leq
        \frac{2^{p+1}(p-1)^{p-1}\Lambda^{p/2}\lambda^{-n/2}}{(R-r)^p}
        |B_R\setminus B_r|,
\]
where the last volume is the Euclidean volume in the chosen chart.
\end{proposition}

\begin{proof}
The weak formulation of the equation is
\begin{equation}\label{eq:B-weak-formulation}
        \int_U |\nabla u|_g^{p-2}\langle\nabla u,\nabla\varphi\rangle_g\,d\mu_g=0,
        \qquad \varphi\in W^{1,p}_0(U).
\end{equation}
Choose a nonnegative cut-off function $\eta\in C_c^\infty(\Omega)$.  Using
$\varphi=u\eta^p$ in \eqref{eq:B-weak-formulation}, we get
\begin{equation}\label{eq:B-testing}
        \int_\Omega \eta^p|\nabla u|_g^p\,d\mu_g
        =-p\int_\Omega u\eta^{p-1}|\nabla u|_g^{p-2}
        \langle\nabla u,\nabla\eta\rangle_g\,d\mu_g .
\end{equation}
By the Cauchy--Schwarz inequality with respect to $g$ and by Young's inequality, in the form
\[
        p a^{p-1}b\leq \frac12 a^p+2^{p-1}(p-1)^{p-1}b^p,
        \qquad a,b\geq0,
\]
we obtain
\begin{equation}\label{eq:B-caccioppoli-core}
        \int_\Omega \eta^p|\nabla u|_g^p\,d\mu_g
        \leq
        2^p(p-1)^{p-1}\int_\Omega |u|^p|\nabla\eta|_g^p\,d\mu_g .
\end{equation}
Indeed, the term $\frac12\int_\Omega \eta^p|\nabla u|_g^p\,d\mu_g$ is absorbed into the left-hand
side of \eqref{eq:B-testing}.

We now take $\eta$ such that $0\leq\eta\leq1$, $\eta\equiv1$ on $B_r$,
$\operatorname{supp}\eta\subset B_R$, $\operatorname{supp}|D\eta|\subset B_R\setminus B_r$, and
\[
        |D\eta|\leq\frac{2^{1/p}}{R-r}
\]
in the chosen coordinates.  Such a smooth cut-off is obtained by smoothing the distance cut-off;
the small loss represented by $2^{1/p}$ is immaterial.  Since
\[
        |\nabla\eta|_g^p=(g^{ij}\partial_i\eta\partial_j\eta)^{p/2}
        \leq \Lambda^{p/2}|D\eta|^p,
\]
formula \eqref{eq:B-caccioppoli-core} gives
\[
        \int_{B_r}|\nabla u|_g^p\,d\mu_g
        \leq
        \frac{2^{p+1}(p-1)^{p-1}\Lambda^{p/2}}{(R-r)^p}
        \int_{B_R\setminus B_r}|u|^p\,d\mu_g .
\]
If $|u|\leq1$, the first additional bound follows immediately.  Finally, the lower ellipticity
bound $\lambda |\xi|^2\leq g^{ij}\xi_i\xi_j$ implies that the eigenvalues of $g_{ij}$ are bounded
above by $\lambda^{-1}$, and hence $d\mu_g\leq \lambda^{-n/2}dx$.  This proves the last estimate.
\end{proof}

\begin{remark}[Use in the proof of Lemma \ref{p_uniform_gradient_bound}]\label{rem:B-use-lemma23}
In Lemma~2.3 the proposition is applied with $g=g_\varepsilon$ and $u=v$ on a finite family of
coordinate charts covering the compact hypersurface $\Sigma$.  The metrics $g_\varepsilon$ are uniformly
comparable with the original corner metric on these charts.  Hence there are constants
$0<\lambda_0\leq\Lambda_0<\infty$, independent of $\varepsilon$, such that
\[
        \lambda_0|\xi|^2\leq g_\varepsilon^{ij}\xi_i\xi_j\leq \Lambda_0|\xi|^2
\]
on all the coordinate balls involved.  The Caccioppoli constant depends on the upper bound
$\Lambda_0$ through the estimate of $|\nabla\eta|_{g_\varepsilon}$, whereas the lower bound
$\lambda_0$ enters when the volume of the fixed coordinate annuli is bounded uniformly in
$\varepsilon$.  Since $0\leq v_\varepsilon\leq1$, one obtains
\[
        \int_{\Sigma\times(-\varepsilon/2,\varepsilon/2)}|\nabla v_\varepsilon|_{g_\varepsilon}^p\,d\mu_{g_\varepsilon}
        \leq
        \sum_i
        \frac{2^{p+1}(p-1)^{p-1}\Lambda_0^{p/2}}{R_i^p}
        \mu_{g_\varepsilon}\big(B_{2R_i}(q_i)\setminus B_{R_i}(q_i)\big)
        \leq K,
\]
where $K$ is independent of $\varepsilon$.  Thus both ellipticity constants are relevant for the
uniform application, but only the upper ellipticity constant is intrinsic to the Caccioppoli
inequality itself.
\end{remark}

\textbf{Acknowledgements} V.A and L.M are members of the Gruppo Nazionale per l’Analisi Matematica, la Probabilit\'{a} e le loro Applicazioni (GNAMPA), which is part of the Istituto Nazionale di Alta Matematica (INdAM) and gratefully acknowledge the support by the MUR PRIN-2022JJ8KER grant “Contemporary perspectives on geometry and gravity”.

%
%

\printbibliography

\end{document}